\documentclass{article}
\usepackage{graphicx} 
\usepackage{url}
\usepackage{amsmath, amssymb, amsthm}
\usepackage{tikz}
\usepackage{mathdots}
\usepackage{mathtools}
\usepackage{fancybox}
\usepackage{enumitem}
\usetikzlibrary{decorations.pathreplacing,calligraphy}

\newtheorem{theorem}{Theorem}
\newtheorem{definition}[theorem]{Definition}
\newtheorem{lemma}[theorem]{Lemma}
\newtheorem{cor}[theorem]{Corollary}
\newtheorem{conjecture}[theorem]{Conjecture}
\newtheorem{quest}[theorem]{Question}
\newtheorem{prop}[theorem]{Proposition}
\theoremstyle{remark}
\newtheorem{example}[theorem]{Example}
\newtheorem{remark}[theorem]{Remark}
\newtheorem*{remark*}{Remark}

\newcommand{\fqp}{\frac{q}{p}}

\newcommand{\calP}{\mathcal{P}}

\newcommand{\mk}{m^{(k)}}

\definecolor{light-blue}{HTML}{e9f1f8}
\definecolor{grayish}{RGB}{220,220,220}
\definecolor{lavender}{rgb}{0.5,0,1.0}

\usepackage{hyperref}
\hypersetup{
    colorlinks=true,
    linkcolor=lavender,
    filecolor=blue,      
    urlcolor=teal,
    citecolor=teal
}

\usepackage{xcolor}

\definecolor{Dgreen}{RGB}{2,100,64}

\title{Orderings of Generalized $k$-Markov Numbers}
\author{Esther Banaian and Min Huang}
\date{December 2025}

\begin{document}

\maketitle

\begin{abstract}
A $k$-Markov number is a positive integer that appears in a positive integral solution to the Diophantine equation $x^2 + y^2 + z^2 + k(xy + xz + yz) = (3+3k)xyz$. This equation was introduced by Gyoda and Matsushita. When $k =0$, this definition recovers that of ordinary Markov numbers. The set of $k$-Markov numbers can be indexed by pairs of coprime positive integers. There is a consistent way to label non-coprime pairs with positive integers as well, yielding a larger set of ``generalized $k$-Markov numbers.''

In this paper, we classify lines along which the generalized $k$-Markov numbers grow monotonically, extending work in the ordinary case by Lee-Li-Rabideau-Schiffler and by the second author. We find that, as $k$ grows, the $k$-Markov numbers are more likely to be monotonic along a random line. This gives evidence that a $k$-version of Frobenius' uniqueness conjecture, which has been proposed by Gyoda and Maruyama, could be true.  
\end{abstract}

\section{Introduction}

A \emph{Markov number‌} (or Markoff number) is a positive integer that appears in a solution to the ‌\emph{Markov equation‌}
\[x^2+y^2+z^2=3xyz,\]
first investigated by Markoff in his seminal works of 1879 and 1880. A triple which satisfies the Markov equation is called a \emph{Markov triple}. As a classical subject in number theory, Markov numbers exhibit deep connections to a diverse array of mathematical disciplines, including combinatorics, hyperbolic geometry, Diophantine approximation theory, and cluster algebras.

A driving force in the study of Markov numbers is the uniqueness conjecture, posed by Frobenius in 1913, which states that every Markov number is the largest number in a unique Markov triple. The conjecture remains open today, although interesting solutions to partial results have been plentiful; see \cite{aigner2013Markov,Reutenauer}.

There is a common method for indexing Markov numbers with the numbers in $\mathbb{Q} \cap [0,1]$. One way in which this can be described is by comparing a tree structure associated to each set. First, Markov triples can be ordered using \emph{Vieta jumps}. It is straightforward to verify that, if $(x,y,z)$ is a Markov triple, then so is \[
(x,\frac{x^2 + y^2}{z},y) = (x,3xy-z,y).
\]

Using such replacements, Markov triples populate a nearly 3-regular tree with root $(1,1,1)$. It can be showed that every Markov triple, therefore every Markov number, appears on this tree \cite{markoff1879formes}.

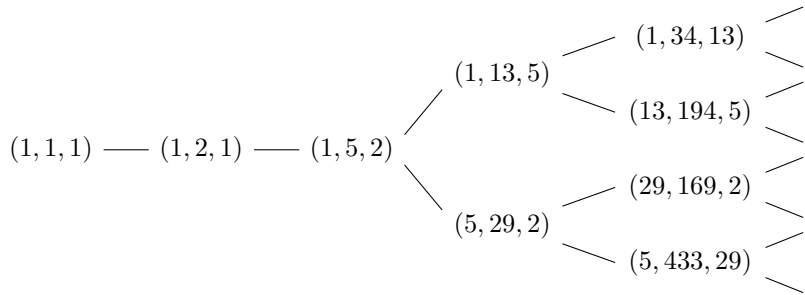
\begin{center}
\begin{tikzpicture}
\node[] at (-1,0){$(1,1,1)$};
\node[] at (1,0){$(1,2,1)$};
\node[] at (3,0){$(1,5,2)$};
\node[] at (5,1){$(1,13,5)$};
\node[] at (5,-1){$(5,29,2)$};
\node[] at (7.5,1.5){$(1,34,13)$};
\node[] at (7.5,0.5){$(13,194,5)$};
\node[] at (7.5,-0.5){$(29,169,2)$};
\node[] at (7.5,-1.5){$(5,433,29)$};
\draw (0.7-1,0) -- (1.3-1,0);
\draw(2.7-1,0) -- (3.3-1,0);
\draw (4.2-.5,0.2) -- (5.2-1,0.8);
\draw(4.2-.5,-0.2) -- (5.2-1,-0.8);
\draw(5.8,1.25) -- (6.5,1.5);
\draw(5.8,0.75) -- (6.5,0.5);
\draw(5.8,-0.75) -- (6.5,-0.5);
\draw(5.8,-1.25) -- (6.5, -1.5);
\draw(8.5,1.7) -- (9,1.9);
\draw(8.5,1.3) -- (9,1.1);
\draw (8.5,0.7) -- (9,0.9);
\draw(8.5,0.3) -- (9,0.1);
\draw(8.5,-0.3) -- (9,-0.1);
\draw(8.5,-0.7) -- (9,-0.9);
\draw(8.5,-1.3) -- (9,-1.1);
\draw(8.5,-1.7) -- (9,-1.9);
\end{tikzpicture}
\end{center}

The rational numbers in $(\mathbb{Q} \cap [0,1]) \cup \{\frac{-1}{1},\frac10\}$ form an isomorphic tree with root $(\frac{-1}{1},\frac{0}{1},\frac{1}{0})$ with edges determined by using the Farey sum $\frac{q}{p} \oplus \frac{t}{s} := \frac{q+t}{p+s}$. We call the triples in this tree \emph{Farey triples}.

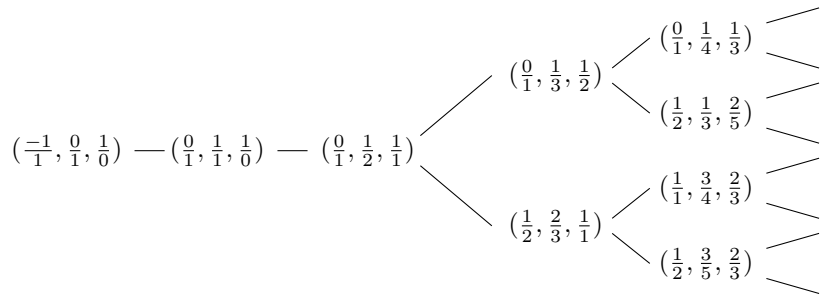
\begin{center}
\begin{tikzpicture}
\node[] at (0-1,0){$(\frac{-1}{1},\frac{0}{1},\frac{1}{0})$};
\node[] at (2-1,0){$(\frac01,  \frac11, \frac{1}{0})$};
\node[] at (4-1,0){$(\frac01,  \frac12, \frac{1}{1})$};
\node[] at (5.5,1){$(\frac01, \frac13, \frac12)$};
\node[] at (5.5,-1){$(\frac12, \frac23, \frac11)$};
\node[] at (7.5,1.5){$(\frac01,\frac14,\frac13)$};
\node[] at (7.5,0.5){$(\frac12, \frac13,\frac25)$};
\node[] at (7.5,-0.5){$(\frac11, \frac34, \frac23)$};
\node[] at (7.5,-1.5){$(\frac12, \frac35, \frac23)$};
\draw (0.9-1,0) -- (1.3-1,0);
\draw(2.8-1,0) -- (3.2-1,0);
\draw (4.2-.5,0.2) -- (4.65,1);
\draw(4.2-.5,-0.2) -- (4.65,-1);
\draw(6.25,1.1) -- (6.75,1.5);
\draw(6.25,0.9) -- (6.75,0.5);
\draw(6.25,-0.9) -- (6.75,-0.5);
\draw(6.25,-1.1) -- (6.75, -1.5);
\draw(8.3,1.7) -- (9,1.9);
\draw(8.3,1.3) -- (9,1.1);
\draw (8.3,0.7) -- (9,0.9);
\draw(8.3,0.3) -- (9,0.1);
\draw(8.3,-0.3) -- (9,-0.1);
\draw(8.3,-0.7) -- (9,-0.9);
\draw(8.3,-1.3) -- (9,-1.1);
\draw(8.3,-1.7) -- (9,-1.9);
\end{tikzpicture}
\end{center}

By identifying Markov and Farey triples via their central elements, one defines the Markov number $m_{\frac{q}{p}}$ by comparing the two vertices related by the canonical isomorphism between these trees. 

The Uniqueness Conjecture would imply that Markov numbers induce a total ordering on $\mathbb{Q} \cap [0,1]$ via the fractional labeling, and indeed these two statements are equivalent. Using this perspective, Aigner gave three weaker conjectures regarding Markov numbers: the constant sum conjecture, the constant numerator conjecture, and the constant denominator conjecture \cite{aigner2013Markov}. For example, the constant sum conjecture states that, if $\gcd(p,q) = 1$ and $0<i<q$ is such that $\gcd(p+i,q-i) = 1$, then $m_{\frac{q}{p}} < m_{\frac{q-i}{p+i}}$.

Using different tools, several different teams have recently shown Aigner's conjectures \cite{LPTV,LLRS,mcshane2021convexity}. In \cite{LLRS}, the authors gave a wider family of relations amongst Markov numbers as follows. Let $(p,q)$ and $(s,t)$ be two pairs of coprime positive integers with $p>q$ and $s>t$. Assume as well that $p < s$. Let $\ell: ax + b$ be the line connecting these two points. The main result of \cite{LLRS} says that, if $a \geq -\frac87$, then $m_{\frac{q}{p}} < m_{\frac{t}{s}}$ and if $a \leq -\frac54$, then $m_{\frac{q}{p}} > m_{\frac{t}{s}}$. The authors conjecture more precise behavior by positing tighter bounds on these regions of increase and decrease as well as the behavior which occurs between. 

These conjectures were settled by the second author in \cite{HuangMonotonicity} as well as by Gaster in \cite{Gaster}. Indeed, both \cite{LPTV} and \cite{LLRS} demonstrated Aigner's conjectures remain true in a larger context by associating a positive integer $m_{(p,q)}$ to each point $(p,q)$ with $0 \leq p \leq q$ in such a way that, if $\gcd(p,q) = 1$, then $m_{(p,q)} = m_\fqp$. The proof in \cite{HuangMonotonicity} works in this wider generality.

The numbers $m_{(p,q)}$ for non-coprime integers $p$ and $q$, sometimes termed ``generalized Markov numbers,'' can be defined using \emph{snake graphs}. A snake graph is a combinatorial tool that was developed to study cluster algebras from surfaces \cite{MusikerSchiffler1,propp2005combinatorics}. It is well-known that Markov numbers can be viewed as specializations of cluster variables in the cluster algebra arising from a once-punctured torus \cite{beineke2011cluster,propp2005combinatorics}. In this correspondence, cluster variables are associated to \emph{arcs} on the once-punctured torus, i.e. non-self-intersecting curves with both endpoints on the puncture. One can associate a snake graph $\mathcal{G}_{\fqp}$ with each Markov number $m_{\fqp}$ in such a way that $\mathcal{G}_{\fqp}$ has $m_{\fqp}$ perfect matchings. The wider family of numbers $m_{(p,q)}$ comes from building snake graphs from a wider set of curves on the once-punctured torus. Indeed, this perspective has proven to be useful in studying orderings of Markov numbers \cite{rabideau2020continued,LLRS}.

Viewing Markov numbers through the lens of cluster algebras suggests natural generalizations. Given a nonnegative integer $k$, Gyoda and Matushita introduced the $k$-Markov equation in \cite{GyodaMatsushita}, \begin{equation}
x^2 + y^2 + z^2 + k(xy + xz + yz) = (3+3k)xyz.
\end{equation}

Prior to this definition, the case for $k = 1$ was studied by the first author and Sen \cite{banaian2024generalization} and Gyoda \cite{GyodaOriginal}. In this setting, a Vieta jump replaces a triple $(x,y,z)$ with, for example, $(x,\frac{x^2 + kxy + y^2}{z},y)$. The form of these replacements resembles mutation in Chekhov and Shapiro's \emph{generalized cluster algebras} \cite{chekhov2014teichmuller}. Therefore, we can study $k$-Markov numbers using the snake graphs for generalized cluster algebras developed by the first author and Kelley in \cite{banaian2020snake}.

The same fractional labeling works for $k$-Markov numbers for any $k$. We accordingly denote a $k$-Markov number by $m^{(k)}_{\fqp}$ and a generalized $k$-Markov number by $m^{(k)}_{(p,q)}$. In Section~\ref{sec:Definition}, we will redefine these quantities in a concrete way. 

The first author recently showed that the $k$-Markov numbers also satisfy Aigner's conjectures. Here, we extend the work and exhibit all lines along with $k$-Markov numbers monotonically increase or decrease. This is the $k$-Markov analogue of the work in \cite{HuangMonotonicity}. Our main result is the following.

\begin{theorem}(Theorem~\ref{thm:Main})\label{thm:Main-intro}
Let $\ell:y=ax+b$ be a line with rational slope and intercept. Let $\alpha = 4k^2 + 12k + 5$, $\beta= 3k^2 + 8k + 6$ and $\delta= 3k^4 + 17k^3 + 34 k^2 + 28k + 8$  and let $A= \frac{2k^2 + 3k + 1 + (1+k)\sqrt{\alpha}}{2(1+2k)\sqrt{\alpha}}$, $B=\frac{(k+1)(k+2)(\beta^2-4)+\delta\sqrt{\beta^2-4}}{(\beta-2)(\beta^2-4)}$. Let $ \mathcal R=\{(p,q)\in \mathbb Z_{\geq0}^2\mid p\geq q\}$. Then the following hold:
\begin{enumerate}
    \item[$(1)$] If $a \geq -\frac{\ln(3+3k)B}{\ln\frac{\beta+\sqrt{\beta^2-4}}{2(3+3k)B}}$,
    then the generalized $k$-Markov numbers are strictly increasing as functions of $x$ along $\ell\cap \mathcal R$. 

    \item[$(2)$] If $a\leq -\frac{\ln\frac{3+2k+\sqrt{\alpha}}{2}}{\ln(3+3k)A}$,
    then the generalized $k$-Markov numbers are strictly decreasing as functions of $x$ along $\ell\cap \mathcal R$. 

    \item[$(3)$] If $-\frac{\ln\frac{3+2k+\sqrt{\alpha}}{2}}{\ln(3+3k)A}< a<-\frac{\ln(3+3k)B}{\ln\frac{\beta+\sqrt{\beta^2-4}}{2(3+3k)B}}$, then for any 
    $b\in \mathbb Q$ such that $|\ell\cap \mathcal R|>2$,
    the generalized $k$-Markov numbers are neither monotone increasing nor monotone decreasing in $x$ along $\ell\cap \mathcal R$. More precisely, they first decrease and then increase as $x$ grows.
\end{enumerate}
\end{theorem}

Notably, the ``gray area'' between the slopes of increase and decrease, as described in Theorem~\ref{thm:Main-intro} (3), becomes arbitrarily small as $k$ increases. Consequently, in the limit, a $k$-version of the Uniqueness Conjecture holds true.

The article is structured as follows. Section~\ref{sec:Preliminaries} contains the necessary prerequisites, including our poset construction and key definitions. In Section~\ref{sec:Ptolemy}, we expand the generalized Ptolemy relation from \cite{banaian2025} to allow some collinear points. This key result is used frequently throughout. We further prepare for the main results in  Section~\ref{sec:Sequences} by studying several sequences of generalized $k$-Markov numbers, including $k$-analogues of Fibonacci numbers and Pell numbers. Our main result is given in Section~\ref{sec:MainResult}, which then concludes with examples and questions for future work.

\section*{Acknowledgements}
This work was partially supported by the German Research Foundation SFB-TRR 358/1 2023 – 491392403 (E.B.) and the National Natural Science Foundation
of China (No.12471023) (M.H.). 

\section{Preliminaries}\label{sec:Preliminaries}

\subsection{Fence Posets}

Throughout the paper, we fix a nonnegative integer $k$. 

Let $\mathcal{L}$ be the lattice in $\mathbb{Z}^2$ consisting of all lines of slope $0,-1,$ and $\infty$ which pass through integral points. We will regard each line as a sequence of line segments between consecutive integral points. 

A \emph{fence poset} is one whose Hasse diagram is a path graph. We will associate one fence poset $\calP_{(p,q)}$ to each point $(p,q) \in \mathbb{Z}^2$. In the next section, we will use these posets to associate a positive integer to each integral point.

First, we will assign a fence poset to any curve in $\mathbb{R}^2$ based on its intersections with $\mathcal{L}$. We will always consider curves to have an orientation.

\begin{definition}\label{def:PosetPpq}

Let $\gamma$ be a curve in $\mathbb{R}^2$. The following process will define a fence poset $(\calP_{\gamma},\succeq)$. 
For clarity, in the construction we will label elements of $\calP_{\gamma}$ by $(\tau,i)$ where $\tau$ is a line segment in $\mathcal{L}$ which $\gamma$ crosses and $0 \leq i \leq k$. 

\begin{enumerate}[label = (\arabic*)]
    \item For every intersection of $\gamma$ with a line segment in $\mathcal{L}$, we will have a chain of $k+1$ elements. If this intersection point lies strictly closer to the endpoint of $\tau$ which lies to the right of $\gamma$, then the chain is of the form $(\tau,0) \succ (\tau,1) \succ \cdots \succ (\tau,k)$. Otherwise, the chain is of the form $(\tau,0) \prec (\tau,1) \prec \cdots \prec (\tau,k)$.
    \item  Suppose $\gamma$ crosses line segments $\tau$ and $\tau'$ consecutively in this order. If the common endpoint of $\tau$ and $\tau'$ lies to the right of $\gamma$, then we introduce the relation $(\tau,k) \succ (\tau',0)$. Otherwise, the common endpoint of $\tau$ and $\tau'$ lies to the left of $\gamma$ and we introduce the relation $(\tau,k) \prec (\tau',0)$
\end{enumerate}

In particular, if $\gamma$ does not intersect any line segments in $\mathcal{L}$, then $\calP_{\gamma} = \emptyset$. 
\end{definition}

We will be interested in the posets associated to a family of arcs, one for each pair of points in $\mathbb{Z}^2$, and we will be particularly interested in those where one of the points is the origin.

If $\gcd(\vert p\vert,\vert q \vert) = 1$, then, the line segment with endpoints $(0,0)$ and $(p,q)$ does not pass through any other integral points. Call this curve $\gamma_{(p,q)}$ and orient it from $(0,0)$ to $(p,q)$. In general, we will sometimes use $\fqp$ in place of $(p,q)$ when $p$ and $q$ are coprime. So, here we may also write $\gamma_\fqp$.

\begin{remark}\label{rem:LeftBias}
The curve $\gamma_{(p,q)}$ may intersect some line segments in $\mathcal{L}$ at exactly their midpoint. Notice that Definition~\ref{def:PosetPpq} treats these as if they are left intersections. We summarize this property by calling the definition ``left-biased.''  
Indeed, we could equivalently use a ``right-based'' definition.
\end{remark}

Now, given a general point $(p,q) \in \mathbb{Z}^2$, let $g = \gcd(\vert p\vert,\vert q\vert)$. Let $p' = \frac{p}{g}$ and $q' = \frac{q}{g}$. The line segment connecting $(0,0)$ and $(p,q)$ can be broken up into $g$ translates of $\gamma_{(p',q')}$. Let $\gamma^L_{(p,q)}$ be the result of taking each such line segment and deforming each intersection with a point of the form $(\frac{n}{2},\frac{m}{2})$ for $n,m \in \mathbb{Z}$ slightly to the left with respect to the orientation on the curve which starts at $(0,0)$. (See Example~\ref{ex:P_{4,2}}.) Define $\gamma^R_{(p,q)}$ similarly but with respect to rightward perturbations. Notice that these perturbations mean that $\gamma_{(p,q)}^L$ and $\gamma_{(p,q)}^R$ intersect more line segments than the line segment between $(0,0)$ and $(p,q)$ and these curves never crosses an line segment from $\mathcal{L}$ at exactly its midpoint. 
For shorthand, we set $\calP_{(p,q)} := \calP_{\gamma_{(p,q)}^L}$; we will later see that our preference for $\gamma_{(p,q)}^L$ has no lasting effect. We will also set $h_{(p,q)} = \vert \calP_{(p,q)} \vert$. 

\begin{example}\label{ex:P_{3,2}}
Here, we illustrate the construction of $\calP_{(3,2)}$ when $k= 1$. First, consider the $\gamma_{(3,2)}$.
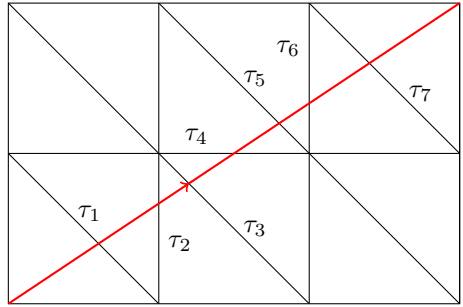
\begin{center}
\begin{tikzpicture}[scale=2]
\draw (0,0) -- (3,0) -- (3,2) -- (0,2) -- (0,0);
\draw (1,0) -- (1,2);
\draw (2,0) -- (2,2);
\draw(3,0) -- (3,2);
\draw(0,1) -- (3,1);
\draw (0,1) -- (1,0);
\draw (0,2) -- (2,0);
\draw(1,2) -- (3,0);
\draw(2,2) -- (3,1);
\draw[thick, red,->] (0,0) -- (1.2,0.8);
\draw[thick, red,] (3,2) -- (1.2,0.8);
\node[right] at (0.4,0.6){$\tau_1$};
\node[right] at (1,0.4){$\tau_2$};
\node[right] at (1.5,0.5){$\tau_3$};
\node[above] at (1.25,1){$\tau_4$};
\node[right] at (1.5,1.5){$\tau_5$};
\node[left] at (2,1.7){$\tau_6$};
\node[right] at (2.6,1.4){$\tau_7$};
\end{tikzpicture}
\end{center}

Condition 1 of Definition~\ref{def:PosetPpq} tells us, for example, that $\calP_{(3,2)}$ has the chain $(\tau_1,0) \succ (\tau_1,1)$ since the intersection point of $\gamma_{\frac23}$ and $\tau_1$ occurs closer to $(1,0)$ than $(0,1)$, and $(1,0)$ lies to the right of $\gamma_{\frac23}$. We have $(\tau_2,0) \prec (\tau_2,1)$ since the intersection point of $\gamma_{\frac23}$ and $\tau_2$ is closer to $(1,1)$, which lies to the left of $\gamma_{\frac23}$. For similar reasons $(\tau_3,0) \prec (\tau_3,1)$. The intersection point of $\gamma_{\frac23}$ and $\tau_4$ lies at the midpoint of $\tau_4$. Since this is not strictly closer to the right endpoint, we treat this like a ``left intersection.''

Condition 2 of Definition~\ref{def:PosetPpq} tells us how to connect these chains. Since the common endpoint of $\tau_1$ and $\tau_2$ lies to the right of $\gamma_{\frac23}$, we have $(\tau_1,1) \succ (\tau_2,0)$. The common endpoint of $\tau_2$ and $\tau_3$ lies to the left of $\gamma_{\frac23}$, so we have $(\tau_2,1) \prec (\tau_3,0)$.

Proceeding as such, we see that $\calP_{(3,2)} = \calP_{\frac23}$ is the following fence poset. 

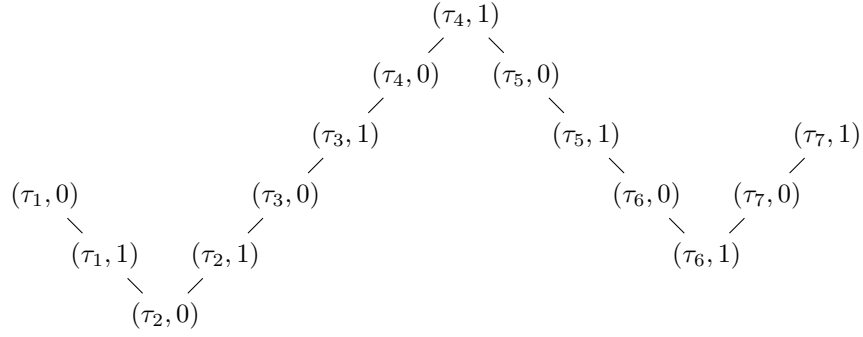
\begin{center}
\begin{tikzpicture}[scale=0.8]
\node(10) at (0,0){$(\tau_1,0)$};
\node(11) at (1,-1){$(\tau_1,1)$};
\node(20) at (2,-2){$(\tau_2,0)$};
\node(21) at (3,-1){$(\tau_2,1)$};
\node(30) at (4,0){$(\tau_3,0)$};
\node(31) at (5,1){$(\tau_3,1)$};
\node(40) at (6,2){$(\tau_4,0)$};
\node(41) at (7,3){$(\tau_4,1)$};
\node(50) at (8,2){$(\tau_5,0)$};
\node(51) at (9,1){$(\tau_5,1)$};
\node(60) at (10,0){$(\tau_6,0)$};
\node(61) at (11,-1){$(\tau_6,1)$};
\node(70) at (12,0){$(\tau_7,0)$};
\node(71) at (13,1){$(\tau_7,1)$};
\draw(10) -- (11);
\draw(11) -- (20);
\draw(20) -- (21);
\draw(21) -- (30);
\draw(30) -- (31);
\draw(31) -- (40);
\draw(40) -- (41);
\draw(41) -- (50);
\draw(50) -- (51);
\draw(51) -- (60);
\draw(60) -- (61);
\draw(61) -- (70);
\draw(70) -- (71);
\end{tikzpicture}
\end{center}

Indeed, one can see that if we treated the intersection at $\tau_4$ as a ``right intersection'', we would get an isomorphic poset. 
\end{example}

\begin{example}\label{ex:P_{4,2}}
Below, we drawn $\gamma_{(4,2)}^L$.  
\begin{center}
\begin{tikzpicture}[scale=2]
\draw (0,0) -- (4,0) -- (4,2) -- (0,2) -- (0,0);
\draw (1,0) -- (1,2);
\draw (2,0) -- (2,2);
\draw(3,0) -- (3,2);
\draw(0,1) -- (4,1);
\draw (0,1) -- (1,0);
\draw (0,2) -- (2,0);
\draw(1,2) -- (3,0);
\draw(2,2) -- (4,0);
\draw(3,2) -- (4,1);
\draw[thick, red] (0,0) -- (0.9,0.45);
\draw[thick, red] (0.9,0.45) to [out = 60, in = 150,looseness=1.5] (1.1,0.55);
\draw[thick,red] (1.1,0.55)-- (1.85,0.9);
\draw[thick, red] (1.85,0.9) to [out = 75, in = 165,looseness=1.5] (2.15,1.1);
\draw[thick,red] (2.15,1.1) -- (2.9,1.45);
\draw[thick, red] (2.9,1.45) to [out = 60, in = 150,looseness=1.5] (3.1,1.55);
\draw[thick,red] (3.1,1.55) -- (4,2);
\node[right] at (0.4,0.6){$\tau_1$};
\node[right] at (1,0.4){$\tau_2$};
\node[right] at (1.5,0.5){$\tau_3$};
\node[above] at (1.25,1){$\tau_4$};
\node[right] at (1.5,1.5){$\tau_5$};
\node[left] at (2,1.7){$\tau_6$};
\node[right] at (2.4,1.6){$\tau_7$};
\node[right] at (3,1.3){$\tau_8$};
\node[right] at (3.5,1.5){$\tau_9$};
\end{tikzpicture}
\end{center}

Then, using the process detailed in Definition~\ref{def:PosetPpq}, we build $\calP_{(4,2)}$.

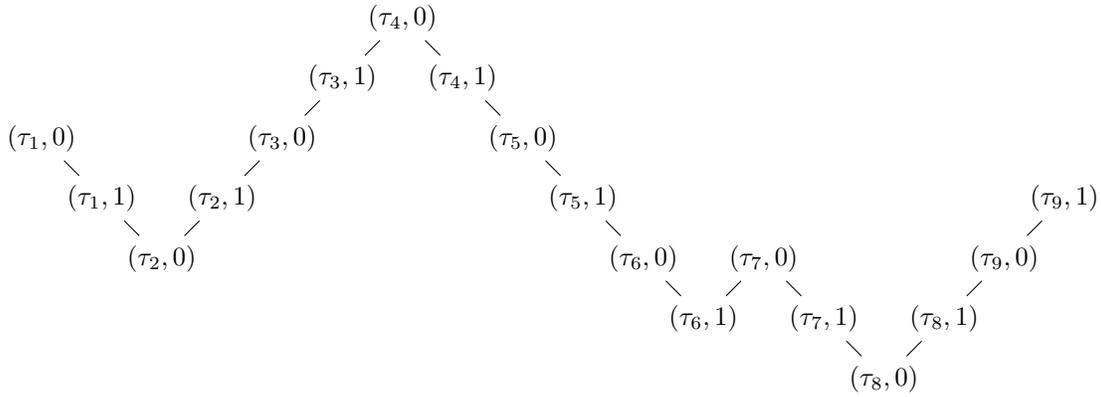
\begin{center}
\begin{tikzpicture}[scale=0.8]
\node(1) at (0,0){$(\tau_1,0)$};
\node(2) at (1,-1){$(\tau_1,1)$};
\node(3) at (2,-2){$(\tau_2,0)$};
\node(4) at (3,-1){$(\tau_2,1)$};
\node(5) at (4,0){$(\tau_3,0)$};
\node(6) at (5,1){$(\tau_3,1)$};
\draw(1) -- (2);
\draw(2) -- (3);
\draw(3) -- (4);
\draw(4) -- (5);
\draw(5) -- (6);
\node(7) at (6,2){$(\tau_4,0)$};
\node(8) at (7,1){$(\tau_4,1)$};
\node(9) at (8,0){$(\tau_5,0)$};
\node(10) at (9,-1){$(\tau_5,1)$};
\node(b) at (10,-2){$(\tau_6,0)$};
\node(c) at (11,-3){$(\tau_6,1)$};
\draw(6) -- (7);
\draw(7) -- (8);
\draw(8) -- (9);
\draw(9) -- (10);
\draw(10) -- (b);
\draw(b) -- (c);
\node(11) at (12,-2){$(\tau_7,0)$};
\node(12) at (13,-3){$(\tau_7,1)$};
\node(13) at (14,-4){$(\tau_8,0)$};
\node(14) at (15,-3){$(\tau_8,1)$};
\node(15) at (16,-2){$(\tau_9,0)$};
\node(16) at (17,-1){$(\tau_9,1)$};
\draw(c)--(11);
\draw(11) -- (12);
\draw(12) -- (13);
\draw(13) -- (14);
\draw(14) -- (15);
\draw(15) -- (16);
\end{tikzpicture}
\end{center}  

Indeed, one can recognize that $\calP_{(4,2)}$ has two subposets which are isomorphic to $\calP_{(2,1)}$, on elements $\{(\tau_i,j):1 \leq i \leq 3, ~0 \leq j \leq 1\}$ and $\{(\tau_i,j):7 \leq i \leq 9,~ 0 \leq j \leq 1\}$. These are connected by a chain which is the result of the perturbation away from the point $(2,1)$. Moreover, one can again see that $\calP_{\gamma_{(4,2)}^R}$ is isomorphic to $\calP_{(4,2)}$. 
\end{example}

We can now easily associate a fence poset to an approximation of a straight line between any two integral points using translation.

\begin{definition}\label{def:PosetBetweenAnyTwoPoints}
Let $A = (p_1,q_1)$ and $B = (p_2,q_2)$ be two points in $\mathbb{Z}$. We define $\calP_{AB}$ to be the poset $\calP_C$ where $C = (p_2-p_1,q_2-q_1)$.
\end{definition}

The following is evident and will be used implicitly.

\begin{lemma}\label{lem:SwapAandB}
There is a poset isomorphism between $\calP_{AB}$ and $\calP_{BA}$ for all $A,B \in \mathbb{Z}^2$.   
\end{lemma}

\subsection{Definition of Generalized $k$-Markov numbers}\label{sec:Definition}

In this section, we will define a type of distance function on $\mathbb{Z}^2$ using our poset construction. This function is known not to be a metric (see \cite[Remark 3.7]{LLRS}). 

We will be particularly interested in the case where we build a poset $\calP_{(p,q)}$ where $(p,q)$ lie in the following set, 
\begin{equation}
   \mathcal R=\{(p,q)\in \mathbb Z_{>0}^2\mid p>q\}.
\end{equation}

An \emph{order ideal} of a poset $(\calP,\succeq)$ is a subset $I \subseteq \calP$ such that if $x \in I$ and $y \preceq x$, then $y \in I$. That is, it is a ``downwards closed set''. Note that this includes the empty set. 

\begin{definition}\label{def:kMarkovDistance}
Let $A,B \in \mathbb{Z}^2$. 
\begin{itemize}
    \item  We define the \emph{$k$-Markov distance} between $A$ and $B$, $\vert AB \vert_k$, to be the number of order ideals of $\calP_{AB}$.
    \item  If $A = (0,0)$ and $B = (p,q) \in \mathbb{Z}^2$, then we write $m^{(k)}_{(p,q)} = \vert AB \vert_k$. If $(p,q) \in \mathcal{R}$, then we call $m^{(k)}_{(p,q)}$ a \emph{generalized $k$-Markov number}.
    \item If $A = (0,0)$, $B = (p,q) \in \mathcal{R}$, and $\gcd(p,q) = 1$,  then we may write $m^{(k)}_{\fqp} = m^{(k)}_{(p,q)}$. In this case, we call $m^{(k)}_{\fqp}$ a \emph{$k$-Markov number}.
\end{itemize}
\end{definition}

One reason for singling out the case $\gcd(p,q) = 1$ is that these numbers appear as specializations of cluster variables in a generalized cluster algebra \cite{BG}.\footnote{Unfortunately, here the word ``generalized'' is being pulled in two different directions. In our setting, the generalization comes from defining numbers $m^{(k)}_{(p,q)}$ for all points in $(p,q) \in \mathcal{R}$, instead of focusing on the coprime case. The fact that we are working with numbers associated to generalized cluster algebras instead of ordinary cluster algebras is encompassed in the term ``$k$-Markov number.''} We stress that $m^{(0)}_{\fqp}$ is the ordinary Markov number labeled by $\fqp$. We will sometimes instead write $m_{\fqp}$ when referencing the ordinary case.

\begin{remark}
There is no obstruction to constructing a poset $\calP_{(p,q)}$ for $(p,q) \notin \mathcal{R}$ or defining a number $m^{(k)}_{(p,q)}$. Indeed, the numbers $m^{(k)}_{(p,0)}$ and $m^{(k)}_{(p,p)}$ will be studied in Section~\ref{sec:Sequences}. However, working with a larger set will immediately introduce repetition. For example, if $p,q > 0$, then $m^{(k)}_{(p,q)} = m^{(k)}_{(q,p)}$.
\end{remark} 

The terminology of a  distance function is justified by the following. 

\begin{prop}\label{prop:Straight}
Let $A,B \in \mathbb{Z}^2$. If $\gamma$ be any curve with endpoints in $A$ and $B$, then $\vert J(\calP_\gamma) \vert > \vert AB \vert_k$.
\end{prop}

\begin{proof}
This was shown for $k=0$ in \cite{LLRS} and for $k > 0$ in \cite{banaian2025}. 
\end{proof}

Connecting a few results gives us an efficient method to count order ideals in fence posets. First, we will define the \emph{shape}  of a fence poset. Given positive integers $c_1,c_2,\ldots,c_n$, let $\ell_j = \sum_{i=1}^j c_i$. A fence poset of shape $(c_1,c_2,\ldots,c_n)$ is one of the following two forms \[
x_1 \succ x_2 \succ \cdots \succ x_{c_1-1} \succ x_{\ell_1} \prec x_{\ell_1+1} \prec \cdots \prec x_{\ell_1 + c_2 - 1} \prec x_{\ell_2}  \succ x_{\ell_2 +1 }\cdots  x_{\ell_n}
\]
or
\[
x_1 \prec x_2 \prec \cdots \prec x_{c_1-1} \prec x_{\ell_1} \succ x_{\ell_1+1} \succ \cdots \succ x_{\ell_1 + c_2 - 1} \succ x_{\ell_2}  \prec \cdots  x_{\ell_n}
\]

For example, the poset in Example~\ref{ex:P_{3,2}} has shape $(3,5,4,2)$, and the poset in Example~\ref{ex:P_{4,2}} has shape $(3,4,5,1,2,3)$. 

Given a finite sequence of positive integers $a_1,a_2,\ldots,a_n$, define the \emph{continued fraction} $[a_1,a_2,\ldots,a_n]$ to be the following rational number,
\[
[a_1,\ldots,a_n] = a_1 + \cfrac{1}{a_2 + \cfrac{1}{a_3 + \cfrac{1}{\ddots + \cfrac{1}{a_n}}}}.
\]

Let $\mathcal{N}[a_1,\ldots,a_n]$ be the numerator of $[a_1,\ldots,a_n]$ as a reduced fraction. 

\begin{lemma}\label{lem:CountOrderIdeals}
The number of order ideals of a fence poset of shape $(c_1,c_2,\ldots,c_n)$ is $\mathcal{N}[c_1,c_2,\ldots,c_n+1]$.
\end{lemma}

\begin{proof}
This is the result of combining \cite[Theorem 5.4]{musiker2013bases}, a bijection between order ideals of a fence poset and perfect matchings of a family of graphs, and \cite[Theorem A]{CSContFrac}, an enumeration formula for the number of perfect matchings of these graphs. 
\end{proof}

Lemma~\ref{lem:CountOrderIdeals} makes the following immediately evident.

\begin{cor}
The $k$-Markov distance between points $A,B \in \mathbb{Z}^2$ is given by $\mathcal{N}[c_1,c_2,\ldots,c_n+1]$ where $(c_1,c_2,\ldots,c_n)$ is the shape of $\calP_{AB}$.
\end{cor}

\begin{example}
As previously noted, the shape of the poset $\calP_{(3,2)}$ is $(3,5,4,2)$. Therefore, $m^{(1)}_{\frac23}$ is the numerator of $[3,5,4,3]$, which is 217. Similarly, $m^{(1)}_{(4,2)}$ is the numerator of $[3,4,5,1,2,3]$, which is 771. By translation, we also for example have $\vert(1,1)(5,3)\vert_1 = 771$. 
\end{example}

\section{Generalized Ptolemy Relation}\label{sec:Ptolemy}

A key tool for the proofs of our main results is the following generalized Ptolemy relation which the $k$-Markov distance function satisfies. 

\begin{prop}\label{prop:ptolemy}
    Let $A, B, C, D \in \mathbb{Z}^2$ be four distinct points such that one of the following holds.
    \begin{enumerate}
        \item[$(a)$] The segments $\overline{AB}$, $\overline{BC}$, $\overline{CD}$, and $\overline{DA}$ 
        form a convex quadrilateral with diagonals $\overline{AC}$ and $\overline{BD}$.
        
        \item[$(b)$] The points $A$, $B$, and $D$ are non-collinear (so that $\triangle ABD$ is non-degenerate) 
        and $C$ lies on the segment $\overline{BD}$.

        \item[$(c)$] The points $A,B,C$, and $D$ are collinear, $B$ lies on $\overline{AC}$, and $C$ lies on $\overline{BD}$.
    \end{enumerate}
    Then,
    \[
        |AC|_k \cdot |BD|_k \geq |AB|_k \cdot |CD|_k + |AD|_k \cdot |BC|_k.
    \]
\end{prop}

Our proof for this critical proposition will require several tools. Part (a) of Proposition~\ref{prop:ptolemy} was shown in \cite{banaian2025} using ``poset skein relations''. In order to show part (b), we will introduce one of these relations. 

In order to concisely reference elements and subposets of a fence poset, we introduce an indexing. If $\calP$ is a fence poset on $h$ elements, we refer to these elements as $\calP(1),\calP(2),\ldots,\calP(h)$ in such a way that $\calP(i)$ has cover relations with $\calP(i-1)$ and $\calP(i+1)$, if these elements exist. Unless $h = 1$, there are two ways to introduce this indexing; informally, these consist of reading the Hasse diagram of the fence poset left-to-right or right-to-left. Using this indexing, we can define subposets $\calP[a,b] = \{\calP(i): a \leq i \leq b\}$. In the following, we define $\calP[b+1,b] = \emptyset$. The phrasing in the following definition is meant to match previous work by the first author on which this is based (see \cite{banaian2025, banaian2024skein}).

\begin{definition}\label{def:ResolutionType1}
Let $\calP_1$ and $\calP_2$ be two fence posets. Let $h_1:= \vert \calP_1 \vert$. Choose an integer $1 \leq j \leq h_1 -1$. Assume without loss of generality that $\calP_1(j) \succ \calP_1(j+1)$. Define the \emph{Type 1 resolution} of $\calP_1$ and $\calP_2$ with respect to $j$ to be the collection of posets $\calP_3,\calP_4,\calP_5,$ and $\calP_6$ defined as follows.

\begin{itemize}
    \item Let $\calP_3$ be the transitive closure of taking the subposets $\calP_1[1,j] \cup \calP_2$ and setting $\calP_1(j) \prec \calP_2(1)$.
    \item Let $u > j+1$ be the smallest integer such that $\calP_1(j+1) \not\succ \calP_1(u)$, if it exists; otherwise, set $u = h_1+1$. Let $\calP_4$ be the subposet $\calP_1[u,h_1]$.
    \item Let $\calP_5$ be the transitive closure of taking the subposets $\calP_1[j+1,h_1] \cup \calP_2$ and setting $\calP_1(j+1) \succ \calP_2(1)$.
    \item Let $v<j$ be the largest integer such that $\calP_1(v) \not \succ \calP_1(j)$, if it exists; otherwise, set $v = 0$. Let $\calP_6$ be the subposet $\calP_1[1,v]$.
\end{itemize}
\end{definition}

Let the set of order ideals of a poset $\calP$ be $J(\calP)$.

\begin{prop}\label{prop:Type1Resolution}
Let $\calP_1$ and $\calP_2$ be two fence posets. Let $1 \leq j \leq \vert \calP_1 \vert - 1$. Let $\calP_3,\calP_4,\calP_5,$ and $\calP_6$ be the Type 1 resolution of $\calP_1$ and $\calP_2$ with respect to $j$. Then, \[
\vert J(\calP_1) \vert \cdot \vert J(\calP_2) \vert  = \vert J(\calP_3)\vert \cdot \vert J(\calP_4) \vert  + \vert J(\calP_5) \vert \cdot \vert J(\calP_6) \vert 
\]
\end{prop}

\begin{proof}
This can be shown by translating the results in \cite{CS13} to fence posets. For the benefit of the reader, we include a proof. Let $M_1 \subseteq J(\calP_1) \times J(\calP_2)$ consist of pairs of order ideals $(I_1,I_2)$ such that $P_1(j+1) \in I_1$ and if $P_2(1) \in I_2$, then $P_1(j) \in I_1$. Let $M_2 \subseteq J(\calP_1) \times J(\calP_2)$ consist of pairs of order ideals $(I_1,I_2)$ such that $P_1(j) \notin I_1$ and if $P_1(j+1) \in I_1$, then $P_2(1) \in I_2$. One can see $M_1 \cup M_2 = J(\calP_1) \times J(\calP_2)$ and $M_1 \cap M_2 = \emptyset$. Moreover, $M_1$ is in bijection with $J(\calP_3) \times J(\calP_4)$ and $M_2$ is in bijection with $J(\calP_5) \times J(\calP_6)$. The claim follows.
\end{proof}

Our next goal is to prove Proposition~\ref{prop:RecurrenceAlongy=q/px}, which is  advantageous to addressing quadruples of points satisfying condition (c) of Proposition~\ref{prop:ptolemy}. This proposition generalizes results from \cite{LLRS} and \cite{banaian2024generalization} which showed the $k=0$ and $k=1$ cases respectively. We first provide some set-up.

Recall, given $n \geq 1$ and coprime integers $p$ and $q$, the curve $\gamma_{(np,nq)}$ consists of $n$ curves which are nearly isotopic to $\gamma_{(p,q)}$ connected by $n-1$ semicircles about integral points. Accordingly the poset $\calP_{(np,nq)}$ consists of $n$ subposets isomorphic to $\calP_{(p,q)}$, connected by chains consisting of $3k+3$ elements. This structure is described in Figure~\ref{fig:SketchPnpnq}, where we denote the $i$-th copy of $\calP_{(p,q)}$ as $R_i$ and the $i$-th chain as $H_i$. Recall Example~\ref{ex:P_{4,2}} also gave an explicit instance of this phenomenon.

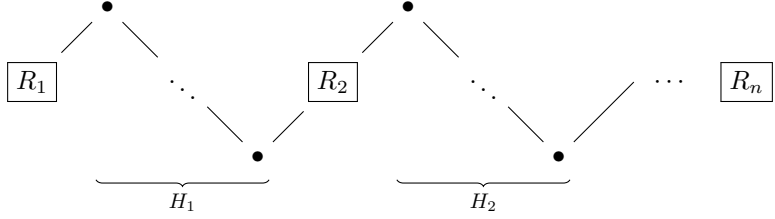
\begin{figure}
\centering
\begin{tikzpicture}
\node(R1) at (0,0){$\boxed{R_1}$};
\node(1) at (1,1){$\bullet$};
\node(dots1) at (2,0){$\ddots$};
\node(2) at (3,-1){$\bullet$};
\draw [decorate,
    decoration = {brace, mirror}] (0.85,-1.3) to node[below, scale = 0.8, yshift = -3pt]{$H_1$}  (3.15,-1.3);
\node(R2) at (4,0){$\boxed{R_2}$};
\node(3) at (5,1){$\bullet$};
\node(dots2) at (6,0){$\ddots$};
\node(4) at (7,-1){$\bullet$};
\draw [decorate,
    decoration = {brace, mirror}] (4.85,-1.3) to node[below, scale = 0.8, yshift = -3pt]{$H_2$}  (7.15,-1.3);
\node(dots3) at (8.5,0){$\cdots$};
\node(Rn) at (9.5,0){$\boxed{R_n}$};
\draw(R1) -- (1);
\draw(1) -- (dots1);
\draw(dots1) -- (2);
\draw(2) -- (R2);
\draw(R2) -- (3);
\draw(3) -- (dots2);
\draw(dots2) -- (4);
\draw(4) -- (8,0);
\end{tikzpicture}
\caption{A sketch of a poset of the form $\calP_{(np,nq)}$ given coprime integers $p$ and $q$ and a positive integer $n$ where each $R_i$ is a subposet isomorphic to $\calP_{(p,q)}$ and each $H_i$ is a chain containing $3k+3$ elements. }\label{fig:SketchPnpnq}
\end{figure}

Our proof of Proposition~\ref{prop:RecurrenceAlongy=q/px} will use another skein relation. This will involve a poset which is the result of adding a relation between the extreme elements of a fence poset. Such posets have been associated to closed curves on a surface (see \cite[Theorem 5.7]{musiker2013bases}). In particular, define $\calP^\circ_{(p,q)}$ be the poset on $h_{(p,q)} + 3k + 3$ elements such that (1) $\calP^\circ_{(p,q)}[3k+4,h_{(p,q)}+3k+3]$ is isomorphic to $\calP_{(p,q)}$, (2) $\calP^\circ_{(p,q)}(3k+3) \prec \calP^\circ_{(p,q)}(3k+4)$, (3) for all $1 < i <  3k+3$, $\calP^\circ_{(p,q)}(i) \succ \calP^\circ_{(p,q)}(i+1)$ and (4) $\calP^\circ_{(p,q)}(h_{(p,q)}+3k+3) \prec \calP^\circ_{(p,q)}(1)$. Conditions (1)-(3) describe a fence poset given by considering $H_1$ and $R_2$ as in Figure~\ref{fig:SketchPnpnq} (see also the definition of $\widetilde{\calP}$ in \cite{BG} as well as \cite[Section 7]{ouguz2025oriented}) while condition (4) makes the Hasse diagram a cycle instead of a path. Since $\calP^\circ_{(p,q)}$ is not a fence poset, we must use slightly different methods to enumerate its order ideals. 

\begin{lemma}\label{lem:NumberOfOrderIdealsCircular}
Let $p$ and $q$ be coprime integers. The number of order ideals of $\calP^\circ_\fqp$ is given by $(3+3k)m_{\frac{p}{q}}^{(k)} - k$.
\end{lemma}

\begin{proof}
We partition the order ideals of $\calP^\circ_\fqp$ based on their support on the first chain.  Let $A_0$ be the set of order ideals $I \in J(\calP^\circ_\fqp)$ such that $I \cap\calP^\circ_\fqp[1, 3k + 3] = \emptyset$ and let $A_1$ be the set of order ideals $I \in J(\calP^\circ_\fqp)$ such that $\calP^\circ_\fqp[1, 3k + 3] \subseteq I$. Given $2 \leq i \leq 3k+3$, $A_i$ be the set of order ideals such that $\calP^\circ_\fqp(i) \in I$ and  $\calP^\circ_\fqp(i-1) \notin I$. 

For all $2 \leq i \leq 3k+3$, $\vert A_i \vert = \vert J(\calP_\fqp) \vert = m^{(k)}_{\frac{q}{p}}$. Based on the structure of $\calP^\circ_{(p,q)}$, we have \[
A_0 = \vert \{I \in J(\calP_\fqp): \calP_{\frac{q}{p}}(1) \notin I\}\vert \qquad \text{ and } \qquad A_1 =  \vert \{I \in J(\calP_{\frac{q}{p}}): \calP_{\frac{q}{p}}(h_\fqp) \in I\}\vert.
\]

Combining \cite[Lemma 8.25]{BG} with the translation to extended posets in \cite[Lemma 33]{banaian2025}, we have that $\vert \{I \in J(\calP_{\frac{q}{p}}): \calP_{\frac{q}{p}}(1) \notin I\}\vert$ is equal to $\vert \{I \in J(\calP_{\frac{q}{p}}):  \calP_{\frac{q}{p}}(h_\fqp) \notin I\}\vert - k$. Since tautologically $\vert \{I \in J(\calP_{\frac{p}{q}}):  \calP_{\frac{p}{q}}(h_{(p,q)}) \notin I\}\vert + \vert \{I \in J(\calP_{\frac{p}{q}}):  \calP_{\frac{p}{q}}(h_{(p,q)}) \in I\}\vert = m_{\frac{p}{q}}^{(k)}$, the claim follows. 
\end{proof}

We are now prepared for our recurrence on generalized $k$-Markov numbers. 

\begin{prop}\label{prop:RecurrenceAlongy=q/px}
 Let $p$ and $q$ be coprime integers.  The sequence $\{m^{(k)}_{np,nq}\}_{n \in \mathbb{N}}$ satisfies \[
m^{(k)}_{(n+2)p,(n+2)q} = ((3+3k)m^{(k)}_\fqp - k)m^{(k)}_{(n+1)p,(n+1)q} - m^{(k)}_{np,nq}
 \]
 with initial conditions $m^{(k)}_{0,0} = 0$ and $m^{(k)}_{p,q} = m^{(k)}_\fqp$. Consequently, if $\eta:= (3 + 3k)m^{(k)}_\fqp -k$, then
 $m^{(k)}_{np,nq}$ is given by \[
\frac{m^{(k)}_\fqp}{\sqrt{\eta^2 - 4}}\bigg(\frac{\eta + \sqrt{\eta^2 - 4}}{2}\bigg)^n - \frac{m_\fqp^{(k)}}{\sqrt{\eta^2 - 4}}\bigg(\frac{\eta - \sqrt{\eta^2 - 4}}{2}\bigg)^n 
 \]   
\end{prop}

\begin{proof}
By Lemma~\ref{lem:NumberOfOrderIdealsCircular}, we can prove the first statement by exhibiting a bijection between $J(\calP_{((n+1)p,(n+1)q)}) \times J(\calP^\circ_{(p,q)})$ and $J(\calP_{((n+2)p,(n+2)q)}) \cup J(\calP_{(np,nq)})$.

As discussed earlier, the poset $\calP_{((n+1)p,(n+1)q)}$ has $n+1$ subposets isomorphic to $\calP_{\fqp}$, labeled $R_1,\ldots,R_{n+1}$. Let $R_j(i)$ denote the element associated to $\calP_{\fqp}(i)$ in the $j$-th copy.  Let $H$ and $R$, without subscripts, denote $\calP^\circ_{\fqp}[1,3k+3]$ and $\calP^\circ_{\fqp}[3k+4,  h_{\fqp}+3k+3]$ respectively. We will similarly let $R(i) = \calP^\circ_{\fqp}(i + 3k + 3)$. 

We partition $J(\calP_{((n+1)p,(n+1)q)}) \times J(\calP^\circ_{(p,q)})$ into several sets and consider maps from each subset into $J(\calP_{((n+2)p,(n+2)q)}) \cup J(\calP_{(np,nq)})$. Let $A_0 \subset J(\calP_{((n+1)p,(n+1)q)}) \times J(\calP^\circ_{(p,q)})$ consist of all pairs $(I_1,I_2)$ such that there exists at least one value $i$ such that either ($R_{n+1}(i) \in I_1$ and $R(i) \in I_2$) or ($R_{n+1}(i) \notin I_1$ and $R(i) \notin I_2$). In this case, we let $j$ be the smallest such value and we say that $I_1$ and $I_2$ \emph{have a switching position} at $j$, as in \cite{banaian2024skein} and inspired by \cite{CS13}. 

Given $(I_1,I_2) \in A_0$ with a switching position at $j$, we define $\Psi(I_1,I_2) = I \in J(\calP_{((n+2)p,(n+2)q)})$ as follows. For all $1 \leq i \leq n$, set $I\vert_{R_i} = I_1 \vert_{R_i}$ and $I \vert_{H_i} = I_1 \vert_{H_i}$. We also set $I \vert_{H_{n+1}} = I_2 \vert_H$. Next, we let \[
I\vert_{R_{n+1}}= \{R_{n+1}(i) : (i \leq j \text{ and } R_{n+1}(i) \in I_1) \text{ or } (i > j \text{ and } R(i) \in I_2) \}
\]
and 
\[
I\vert_{R_{n+2}}= \{R_{n+2}(i) : (i \leq j \text{ and } R(i) \in I_2) \text{ or } (i > j \text{ and } R_{n+1}(i) \in I_1) \}.
\]

One can check that $I$ is indeed an order ideal. The image of $A_0$ under the map $\Psi$ is the set of $I \in J(\calP_{((n+2)p,(n+2)q)})$ such that $I$ has a switching position between $I \vert_{R_{n+1}}$ and $I \vert_{R_{n+2}}$. 

 Now, we partition the complement of $A_0$ in $J(\calP_{((n+1)p,(n+1)q)}) \times J(\calP^\circ_{(p,q)})$. By \cite[Lemma 3]{banaian2024skein}, this consists of pairs $(I_1,I_2)$ such that $I_1 \cap R_{n+1} = R_{n+1}$ and $I_2 \cap R = \emptyset$ or vice versa. 

Let $A_1 \subset J(\calP_{((n+1)p,(n+1)q)}) \times J(\calP^\circ_{(p,q)})$ consist of all pairs $(I_1,I_2)$ such that $I_1 \vert_{R_{n+1}} = R_{n+1}$ and $I_2 \vert_R = \emptyset$. Given $(I_1,I_2) \in A_1$, define $\Psi(I_1,I_2)$ to be $I \in J(\calP_{((n+2)p,(n+2)q)})$ such that $I \vert_{R_i} = I_1 \vert_{R_i}$ for all $1 \leq i \leq n+1$, $I \vert_{H_i} = I_1 \vert_{H_i}$ for all $1 \leq i \leq n$ and $I \vert_{H_{n+1}} = I_2 \vert_{H}$. The image of $A_1$ under the map $\Psi$ consists of all $I \in J(\calP_{((n+2)p,(n+2)q)})$ such that $I \vert_{R_{n+1}} = R_{n+1}$, $I \vert_{R_{n+2}} = \emptyset$, and the maximal element of $H_{n+1}$ is not in $I$.

Let $A_2 \subset J(\calP_{((n+1)p,(n+1)q)}) \times J(\calP^\circ_{(p,q)})$ consist of all pairs $(I_1,I_2)$ such that $I_1 \vert_{R_{n+1}} = \emptyset$, $I_2 \vert_R = R$, and the maximal element of $H$ is not in $I_2$. Given $(I_1,I_2) \in A_2$, define $\Psi(I_1,I_2)$ to be $I \in J(\calP_{((n+2)p,(n+2)q)})$ such that $I \vert_{R_i} = I_1 \vert_{R_i}$ for all $1 \leq i \leq n+1$, $I \vert_{H_i} = I_1 \vert_{H_i}$ for all $1 \leq i \leq n$, $I \vert_{R_{n+2}} = I_2 \vert_{R_{n+2}}$, and $I \vert_{H_{n+1}} = I_2 \vert_{H}$. The image of $A_2$ under the map $\Psi$ consists of all $I \in J(\calP_{((n+2)p,(n+2)q)})$ such that $I \vert_{R_{n+1}} = \emptyset$ and $I \vert_{R_{n+2}} = R_{n+2}$.

Let $A_3 \subset J(\calP_{((n+1)p,(n+1)q)}) \times J(\calP^\circ_{(p,q)})$ consist of all pairs $(I_1,I_2)$ such that $I_1 \vert_{R_{n+1}} = \emptyset$, $I_2 \vert_R = R$, the maximal element of $H$ is in $I_2$, and the minimal element of $H_n$ is in $I_1$.  Given $(I_1,I_2) \in A_3$, define $\Psi(I_1,I_2)$ to be $I \in J(\calP_{((n+2)p,(n+2)q)})$ such that $I \vert_{R_i} = I_1 \vert_{R_i}$ for all $1 \leq i \leq n$, $I \vert_{H_i} = I_1 \vert_{H_i}$ for all $1 \leq i \leq n$, $I \vert_{R_{n+1}} = I_2 \vert_{R}$, and $I \vert_{H_{n+1}} = I_2 \vert_{H}$. The image of $A_3$ under the map $\Psi$ consists of all $I \in J(\calP_{((n+2)p,(n+2)q)})$ such that $I \vert_{R_{n+1}} = R_{n+1}$, $I \vert_{R_{n+2}} = \emptyset$, and the maximal element of $H_{n+1}$ is in $I$.

Finally, let $A_4 \subset J(\calP_{((n+1)p,(n+1)q)}) \times J(\calP^\circ_{(p,q)})$ consist of all pairs $(I_1,I_2)$ such that $I_1 \vert_{R_{n+1}} = \emptyset$, $I_2 \vert_R = R$, the maximal element of $H$ is in $I_2$, and the minimal element of $H_n$ is not in $I_1$. Alternatively, $A_4$ can be described as all pairs $(I_1,I_2)$ such that  $I_1 \vert_{H_n \cup R_{n+1}} = \emptyset$ and $I_2 = \calP^\circ_{(p,q)}$. Any pair $(I_1,I_2) \in A_4$ can naturally be sent to an order ideal $I \in J(\calP_{(np,nq)})$ given by restricting $I_1$ to $R_1 \cup H_1 \cup \cdots \cup R_n$.

One can observe that $\Psi$ is injective. Since $A_0,\ldots,A_4$ partition $J(\calP_{((n+1)p,(n+1)q)}) \times J(\calP^\circ_{(p,q)})$, and the images partition $J(\calP_{((n+2)p,(n+2)q)}) \cup J(\calP_{(np,nq)})$, we conclude that $\Psi$ is indeed bijective and the first statement follows. One can prove the second statement using a straightforward argument involving recurrence relations.
\end{proof}

\begin{proof}[Proof of Proposition~\ref{prop:ptolemy}]
If $\overline{AB}$, $\overline{BC}$, $\overline{CD}$, and $\overline{DA}$  satisfy condition (a), then the statement follows from \cite[Proposition 49]{banaian2025}. 

Next, suppose these four line segments satisfy condition (b). Assume that $B$ and $D$ are labeled such that $\gamma_{CA}$ intersects $\gamma_{BD}$; if this was not true, then $\gamma_{CA}$ would intersect $\gamma_{DB}$. 

Notice that the last line segment from $\mathcal{L}$ which $\gamma_{CA}$ intersects, call it $\tau_1$, forms a triangle with two consecutive line segments which $\gamma_{BD}$ intersects. Let these two consecutive line segments be $\tau_2$ and $\tau_3$, labeled such that $\tau_3$ follows $\tau_2$ in counterclockwise order. Recall this implies that $(\tau_2,k)$ covers $(\tau_3,0)$. Let $j$ be such that $\calP_{BD}(j) =(\tau_2,k)$.

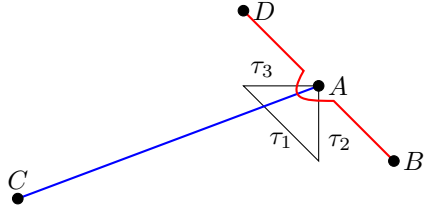
\begin{center}
\begin{tikzpicture}
\node[right] at (3,1){$B$};
\node[right] at (2,2){$A$};
\node[right] at (1,3){$D$};
\draw(2,2) -- (2,1) -- (1,2) -- (2,2);
\draw[thick, blue] (-2,0.5) -- (2,2);
\draw[red,thick] (3,1) -- (2.2,1.8) to [out = 180, in = 240, looseness = 2] (1.8,2.2) -- (1,3);
\draw[fill = black] (-2,0.5) circle [radius=2pt];
\node[above] at (-2,0.5){$C$};
\draw[fill = black] (2,2) circle [radius=2pt];
\draw[fill = black] (3,1) circle [radius=2pt];
\draw[fill = black] (1,3) circle [radius=2pt];
\node[below] at (1.5,1.5){$\tau_1$};
\node[above] at (1.25,2){$\tau_3$};
\node[right] at (2,1.25){$\tau_2$};
\end{tikzpicture}
\end{center}

Let $\calP_3,\calP_4,\calP_5,$ and $\calP_6$ be the resolution of $\calP_1 = \calP_{BD}$ and $\calP_2 = \calP_{CA}$ with respect to $j$ as in Definition~\ref{def:ResolutionType1}. The key to the proof is recognizing that there are curves $\gamma_3,\gamma_4,\gamma_5,$ and $\gamma_6$ such that $\gamma_3$ has endpoints in $A$ and $B$;  $\gamma_4$ has endpoints in $C$ and $D$;  $\gamma_5$ has endpoints in $A$ and $D$;  $\gamma_6$ has endpoints in $B$ and $C$; and $\calP_i = \calP_{\gamma_i}$. 

Indeed, we could recognize $\gamma_3$ to be the result of following $\gamma_{BD}$ in the positive direction until its intersection point with $\gamma_{CA}$, then following $\gamma_{CA}$. Let $\gamma_4'$ be the result of first following $\gamma_{CA}$ then $\gamma_{BD}$, both in the positive direction and switching at their intersection point. Then, $\gamma_4$ is the result of pulling $\gamma_4$ tight so that it does not cross any line segment incident to $C$. The curves $\gamma_5$ and $\gamma_6$ can be built in an analogous way. 

Therefore, by using first Proposition~\ref{prop:Type1Resolution} and then Proposition~\ref{prop:Straight}, we have \[
\vert J(\calP_{AC}) \vert \cdot \vert J(\calP_{BD}) \vert = \vert J(\calP_3) \vert \cdot \vert J(\calP_4) \vert +  \vert J(\calP_5) \vert \cdot \vert J(\calP_6) \vert \geq \vert J(\calP_{AB}) \vert \cdot \vert J(\calP_{BC}) \vert + \vert J(\calP_{AB}) \vert \cdot \vert J(\calP_{BC}) \vert,
\]

and the claim follows from the definition of $\vert AB \vert_k$.

Finally, suppose these four line segments satisfy condition (c). Then, there exist integers $p$ and $q$ and a rational number $b$ such that $\gcd(\vert p \vert,\vert q \vert) = 1$ and $A,B,C,$ and $D$ all lie on the line $\ell:y= \frac{q}{p}x + b$. 

Let $s,t,u$ be nonnegative integers such that there are $s$ integral points on $\ell$ between $A$ and $B$, $t$ integral points between $B$ and $C$, and $u$ integral points between $C$ and $D$. We establish the following terms, \[\eta_{(p,q)} := (3+3k)m_{(p,k)}^{(k)} - k \qquad 
\alpha_{(p,q)}:= \frac{m_{(p,q)}^{(k)}}{\sqrt{\eta_{(p,q)}^2 - 4}}, \qquad \beta_{(p,q)} = \frac{\eta_{(p,q)} + \sqrt{\eta_{(p,q)}^2-4}}{4}.
\]
Then, a direct consequence of Proposition~\ref{prop:RecurrenceAlongy=q/px} is that \begin{align*}
\vert AC \vert_k \cdot \vert BD \vert_k &= \bigg(\alpha_{(p,q)}\big(\beta_{(p,q)}^{s+t} - \beta_{(p,q)}^{-s-t}\big)\bigg)\cdot \bigg(\alpha_{(p,q)}\big(\beta_{(p,q)}^{t+u} - \beta_{(p,q)}^{-t-u}\big)\bigg)\\
&=\alpha_{(p,q)}^2\big(\beta_{(p,q)}^{s+2tu} + \beta_{(p,q)}^{-s-2t-u} - \beta_{(p,q)}^{s-u} - \beta_{(p,q)}^{u-s}\big)
\end{align*}

One can similarly compute the right hand side of the desired equation in terms of $\alpha, \beta, s,t,$ and $u$, yielding
\begin{align*}
\vert AB \vert_k \cdot \vert CD \vert_k = \alpha_{(p,q)}^2\big(\beta_{(p,q)}^{s+u} + \beta_{(p,q)}^{-s-u} - \beta_{(p,q)}^{s-u} - \beta_{(p,q)}^{u-s}\big)
\end{align*}
and 
\begin{align*}
\vert AD \vert_k \cdot \vert BC \vert_k = \alpha_{(p,q)}^2\big(\beta_{(p,q)}^{s+2t + u} + \beta_{(p,q)}^{-s-2t -u} - \beta_{(p,q)}^{s+u} - \beta_{(p,q)}^{-s-u}\big).
\end{align*}

Therefore, we can conclude in this case that $\vert AC \vert_k \cdot \vert BD \vert_k = \vert AB \vert_k \cdot \vert CD \vert_k + \vert AD \vert_k \cdot \vert BC \vert_k$.
\end{proof}

\section{Special Subsequences}\label{sec:Sequences}

The region $\mathcal{R}$ is bordered by the lines $y =0$ and $y = x$.  It will be advantageous to consider the generalized $k$-Markov numbers along and near these lines. 

We begin by considering the $k$-Markov numbers sequences labeled by the sequences $\{\frac{1}{n}\}_{n \geq 0}$ and $\{\frac{n}{n+1}\}_{n \geq 0}$. It is well-known that the numbers $m^{(0)}_{\frac1n}$ are odd-indexed Fibonacci numbers and $m^{(0)}_{\frac{n}{n+1}}$ are odd-indexed Pell numbers. When $k = 0$, Proposition~\ref{prop:kMarkovFibAndPell} recovers well-known recurrences and formulae concerning these numbers.

In the following, given two infinite sequences $\{a_n\}$ and $\{b_n\}$, we write $a_n \sim b_n$ if $\lim_{n \to \infty} \frac{a_n}{b_n} = 1$.

\begin{prop}\label{prop:kMarkovFibAndPell}
Let $k \geq 0$.
\begin{enumerate}[label = (\arabic*)]
    \item The sequence $\{f_n\}$ with $f_n:= m_{\frac1n}^{(k)}$ is given by initial conditions $f_0 = 1, f_1 = k+2$ and for all $n \geq 2$, \[
f_{n} = ( 2k+3)f_{n-1} - f_{n-2} - k.
    \]
Consequently, if $\alpha:= 4k^2 + 12k + 5$, $f_n$ is given by \[
\frac{k}{1+2k} + \frac{2k^2 + 3k + 1 + (1+k)\sqrt{\alpha}}{2(1+2k)\sqrt{\alpha}} \bigg(\frac{3 + 2k + \sqrt{\alpha}}{2}\bigg)^n + \frac{-(2k^2 + 3k + 1) + (1+k)\sqrt{\alpha}}{2(1+2k)\sqrt{\alpha}} \bigg(\frac{3 + 2k - \sqrt{\alpha}}{2}\bigg)^n
\]
so that asymptotically,\[
f_n \sim \frac{2k^2 + 3k + 1 + (1+k)\sqrt{\alpha}}{2(1+2k)\sqrt{\alpha}} \bigg(\frac{3 + 2k + \sqrt{\alpha}}{2}\bigg)^n.
\]
\item The sequence $\{g_n\}$ with $g_n:= m_{\frac{n}{n+1}}^{(k)}$ is given by initial conditions $g_0 = 1, g_1 = 2k^2 + 6k + 5$ and for all $n \geq 2$, \[
g_{n} = (3k^2 + 8k + 6 )g_{n-1} - g_{n-2} - k(k+2).
    \]
Consequently, if $\beta:= 3k^2 + 8k + 6$ and $\delta:= 3k^4 + 17k^3 + 34 k^2 + 28k + 8$, $g_n$ is given by \begin{align*}
\frac{k}{3k+2} &+ \frac{(k+1)(k+2)(\beta^2 - 4) + \delta \sqrt{\beta^2 - 4}}{(\beta-2)(\beta^2 - 4)}\bigg(\frac{\beta + \sqrt{\beta^2 - 4}}{2}\bigg)^n\\& + \frac{(k+1)(k+2)(\beta^2 - 4) -\delta \sqrt{\beta^2 - 4}}{(\beta-2)(\beta^2 - 4)}\bigg(\frac{\beta - \sqrt{\beta^2 - 4}}{2}\bigg)^n
\end{align*}
so that asymptotically, \[
g_n \sim \frac{(k+1)(k+2)(\beta^2 - 4) + \delta \sqrt{\beta^2 - 4}}{(\beta-2)(\beta^2 - 4)}\bigg(\frac{\beta + \sqrt{\beta^2 - 4}}{2}\bigg)^n.
\]
\end{enumerate}
\end{prop}

\begin{proof}
First, consider part (1). Since for $n \geq 2$, $(\frac{0}{1},\frac{1}{n-1},\frac{1}{n-2})$ and  $(\frac{0}{1},\frac{1}{n},\frac{1}{n-1})$ are adjacent Farey triples and $m_{\frac01}^{(k)} = 1$, by \cite[Theorem 1.1]{GyodaMatsushita} we have the following relationship among $f_{n-2},f_{n-1},$ and $f_n$, \begin{equation}
f_n = \frac{f_{n-1}^2 + kf_{n-1} + 1}{f_{n-2}}.\label{eq:MutateFibonacci}
\end{equation}

Since $(1,f_{n-1},f_{n-2})$ is a $k$-Markov triple, we have \[
f_{n-1}^2 + kf_{n-1} + 1 = (2k+3) f_{n-2}f_{n-1} - f_{n-2}^2 - kf_{n-2},
\]

and by substituting the right hand side into Equation~\ref{eq:MutateFibonacci} the first part of the claim follows. The remainder of the statement can be shown with standard methods involving recurrence relations. Part (2) can be proven similarly, using Farey triples $(\frac{n-1}{n},\frac{n}{n+1},\frac11)$. 

\end{proof}

In \cite{HuangMonotonicity}, the second author showed the even-indexed Fibonacci and Pell numbers coincide with generalized Markov numbers $\{m^{(0)}_{(n,0)}\}_{n \geq 0}$ and $\{m^{(0)}_{(n,n)}\}_{n \geq 0}$ respectively. We will consider the analogue of these sequences for general $k$.  First, we introduce a general recurrence for generalized Markov numbers along a fixed line.

\begin{lemma}\label{lem:Slope0and1}
\begin{enumerate}[label = (\arabic*)]
    \item The sequence $m^{(k)}_{(n,0)}$ is given by initial conditions $m^{(k)}_{(0,0)} = 0, m^{(k)}_{(1,0)} = 1$ and for all $n \geq 2$, \[
m^{(k)}_{(n,0)} = (2k+3) m^{(k)}_{(n-1,0)} - \mk_{(n-2,0)}.
\]
\item The sequence $m^{(k)}_{(n,n)}$ is given by initial conditions $m^{(k)}_{(0,0)} = 0, m^{(k)}_{(1,1)} = k+2$ and for all $n \geq 2$, \[
m^{(k)}_{(n,n)} = (3k^2 + 8k + 6) m^{(k)}_{(n-1,n-1)} - \mk_{(n-2,n-2)}.
\]
\end{enumerate}
\end{lemma}

\begin{proof}
Both parts are consequences of  Proposition~\ref{prop:RecurrenceAlongy=q/px}. 
\end{proof}

In the ordinary case, the numbers $m^{(0)}_{(n,1)}$ and $m^{(0)}_{(n,0)}$ intertwine as the Fibonnaci numbers while the numbers $m^{(0)}_{(n,n-1)}$ and $m^{(0)}_{(n,n)}$ intertwine as the Pell numbers. We show an analogous relationship in our setting. 

\begin{prop}\label{prop:Intertwine}
For all $n \geq 1$, we have  \[
\mk_{(n,1)} = (k+1) \mk_{(n,0)} + \mk_{(n-1,1)} 
\]
and
\[
\mk_{(n+1,n)} = 2(k+1)\mk_{(n,n)} + \mk_{(n,n-1)}.
\]

\end{prop}

\begin{proof}
One can check directly that each statement is true for $n = 1,2$, and show the inductive step using Proposition~\ref{prop:kMarkovFibAndPell} and Lemma~\ref{lem:Slope0and1}.
\end{proof}

Now, with our understanding of the generalized $k$-Markov numbers near the boundary of $\mathcal{R}$, we define two important bounds. 
For any two coprime positive integers  $ a_1 $  and  $ a_2 $ , define the constants  $S(a_1,a_2)_\pm$ by
\begin{equation}\label{eq:a_pm_def}
S(a_1,a_2)_- := \lim_{q \to \infty} \frac{m^{(k)}_{\bigl(q(1+a_1)+a_2,\; 1\bigr)}}{m^{(k)}_{\bigl(q(1+a_1),\; 1+a_1\bigr)}},
\qquad
S(a_1,a_2)_+ := \lim_{q \to \infty} \frac{m^{(k)}_{\bigl((q+1)(1+a_1+a_2),\; q(1+a_1+a_2)\bigr)}}{m^{(k)}_{\bigl(q(1+a_1+a_2)+a_1+1,\; q(1+a_1+a_2)+a_1\bigr)}}.
\end{equation}

\begin{lemma}
Let $\alpha = 4k^2 + 12k + 5$, $\beta= 3k^2 + 8k + 6$ and $\delta= 3k^4 + 17k^3 + 34 k^2 + 28k + 8$ as in Proposition~\ref{prop:kMarkovFibAndPell} and let $A:= \frac{2k^2 + 3k + 1 + (1+k)\sqrt{\alpha}}{2(1+2k)\sqrt{\alpha}}$, $B:=\frac{(k+1)(k+2)(\beta^2-4)+\delta\sqrt{\beta^2-4}}{(\beta-2)(\beta^2-4)}$.
Then, for any two coprime positive integers $a_1$ and $a_2$, we have  
\begin{equation}
S(a_1,a_2)_-= ((3+3k)A)^{-a_1}\bigg(\frac{3+2k+\sqrt{\alpha}}{2}\bigg)^{a_2}
\end{equation}
and
\begin{equation}\label{eq:s+}
S(a_1,a_2)_+=((3+3k)B)^{a_1 + a_2} \bigg(\frac{\beta + \sqrt{\beta^2-4}}{2}\bigg)^{-a_1}. 
\end{equation}

\end{lemma}

\begin{proof}
We first look at $S(a_1,a_2)_-$. By Proposition~\ref{prop:kMarkovFibAndPell}, we have \[
\mk_{(q(1+a_1)+a_2,1)} \sim A \bigg(\frac{3 + 2k + \sqrt{\alpha}}{2}\bigg)^{q(1+a_1)+a_2}.
\]
Next, $\mk_{(q(1+a_1),1+a_1)}$ is the $(1+a_1)$-th term in the sequence $\mk_{\frac1q} = \mk_{(q,1)},\mk_{(2q,2)},\ldots$. By Proposition~\ref{prop:RecurrenceAlongy=q/px}, if $\eta = (3+3k)\mk_{\frac1q}-k$, we have 
\[
\mk_{(q(1+a_1),1+a_1)} = \frac{\mk_{\frac1q}}{\sqrt{\eta^2-4}} \bigg(\frac{\eta + \sqrt{\eta^2 - 4}}{2}\bigg)^{1+a_1} - \frac{\mk_{\frac1q}}{\sqrt{\eta^2-4}} \bigg(\frac{\eta - \sqrt{\eta^2 - 4}}{2}\bigg)^{1+a_1} 
\]
Now, as a function in positive integers $q$, we have 
\[
\lim_{q \to \infty} 
\frac{
    \frac{\mk_{\frac1q}}{\sqrt{\eta^2-4}} \left(\frac{\eta - \sqrt{\eta^2 - 4}}{2}\right)^{1+a_1}
}{
    \frac{\mk_{\frac1q}}{\sqrt{\eta^2-4}} \left(\frac{\eta + \sqrt{\eta^2 - 4}}{2}\right)^{1+a_1}
} = 0
\]

so that \[
\mk_{(q(1+a_1),1+a_1)} \sim \frac{\mk_{\frac1q}}{\sqrt{\eta^2-4}} \bigg(\frac{\eta + \sqrt{\eta^2 - 4}}{2}\bigg)^{1+a_1} \sim \frac{\mk_{\frac1q} \eta^{1+a_1}}{\sqrt{\eta^2-4}} \sim (3+3k)^{a_1} (\mk_{\frac1q})^{1+a_1}
\]
and we conclude using Proposition~\ref{prop:kMarkovFibAndPell} again \[
\mk_{(q(1+a_1),1+a_1)} \sim (3+3k)^{a_1} A^{a_1+1} \bigg(\frac{3 + 2k + \sqrt{\alpha}}{2}\bigg)^{q(1+a_1)}.
\]
At this point, the final statement follows readily.

The identity~\eqref{eq:s+} for $S(a_1,a_2)_+$ can be proved by an analogous argument.
\end{proof}

\section{Monotonicity of the generalized $k$-Markov numbers}\label{sec:MainResult}

\subsection{Monotonicity along horizontal and vertical lines}

In this subsection, we study the monotonicity of the generalized $k$-Markov numbers along horizontal lines $\ell: y = b$ and vertical lines $\ell: x = n$, where $b, n \in \mathbb{Z}_{>0}$.  This extends the main result of the first author \cite{banaian2025} to generalized $k$-Markov numbers. 

For any $(p,q) \in \mathbb{Z}^2 $ such that $p \geq 0, q \geq 0$ and $(p,q) \neq (0,0)$,  
 we define the horizontal and vertical ratios
\[
    h(p,q) := \frac{m^{(k)}_{(p+1,q)}}{m^{(k)}_{(p,q)}}, \qquad 
    v(p,q) := \frac{m^{(k)}_{(p,q+1)}}{m^{(k)}_{(p,q)}}.
\]

\begin{lemma}\label{lem:hv1}
If $(p,q) \in \mathbb{Z}^2 $ is such that $p \geq 0$ and  $q > 0$, then
\[
    h(p,q) = \frac{m^{(k)}_{(p+1,q)}}{m^{(k)}_{(p,q)}} 
    < 
    \frac{m^{(k)}_{(p+2,q)}}{m^{(k)}_{(p+1,q)}} = h(p+1,q).
\]
and if $(p,q) \in \mathbb{Z}^2 $ is such that $p > 0$ and  $q \geq 0$, then
\[
    v(p,q) = \frac{m^{(k)}_{(p,q+1)}}{m^{(k)}_{(p,q)}} 
    < 
    \frac{m^{(k)}_{(p,q+2)}}{m^{(k)}_{(p,q+1)}} = v(p,q+1).
\]
\end{lemma}

\begin{proof}
 Consider the points $A = (0,0)$, $B = (p+1,q)$, $C = (p+2,q)$, and $D = (1,0)$ where $(p,q)$ is as in the statement.  
Since $q > 0$, these four points form a convex quadrilateral satisfying condition~(a) of Proposition~\ref{prop:ptolemy}, and from this Proposition we have
\[
    |AC|_k \cdot |BD|_k > |AB|_k \cdot |CD|_k,
\]
since all terms are positive and $\vert AD \vert_k \cdot \vert BC \vert_k > 0$. 
Using the identities
\[
    |AC|_k = m^{(k)}_{(p+2,q)}, \quad 
    |BD|_k = m^{(k)}_{(p,q)}, \quad 
    |AB|_k = |CD|_k = m^{(k)}_{(p+1,q)},
\]
we can rearrange our inequality to find
\[
    \frac{m^{(k)}_{(p+1,q)}}{m^{(k)}_{(p,q)}} 
    < 
    \frac{m^{(k)}_{(p+2,q)}}{m^{(k)}_{(p+1,q)}},
\]
which is exactly the inequality $h(p,q) < h(p+1,q)$.

For the second inequality, set $A = (0,0)$, $B = (p,q+1)$, $C = (p,q+2)$, and $D = (0,1)$.  
Again, since $p > 0$, these points form a convex quadrilateral satisfying condition~(a) of Proposition~\ref{prop:ptolemy}.  
The Ptolemy inequality gives
\[
    |AC|_k \cdot |BD|_k > |AB|_k \cdot |CD|_k,
\]
and with
\[
    |AC|_k = m^{(k)}_{(p,q+2)}, \quad 
    |BD|_k = m^{(k)}_{(p,q)}, \quad 
    |AB|_k = |CD|_k = m^{(k)}_{(p,q+1)},
\]
we deduce
\[
    \frac{m^{(k)}_{(p,q+1)}}{m^{(k)}_{(p,q)}} 
    < 
    \frac{m^{(k)}_{(p,q+2)}}{m^{(k)}_{(p,q+1)}},
\]
i.e., $v(p,q) < v(p,q+1)$, as required.
\end{proof}

Notice that if we considered the case $q = 0$ in part (1) of Lemma~\ref{lem:hv1}, then the statement would be false. For example, if $k = 0$, so that $m_{(n,0)} = F_{2n}$, i.e. the $(2n)$-th Fibonacci number, we have the well-known identity $F_{2n+2}F_{2n-2} - F_{2n}^2 = -1$, implying  the $h(p,0)$ sequence decreases when $k =0 $.

\begin{lemma}\label{lem:hv2}
If $(p,q) \in \mathbb{Z}^2 $ is such that $p \geq 0, q \geq 0$ and $(p,q) \neq (0,0)$, then
\[
    m^{(k)}_{(p+1,q)} \, m^{(k)}_{(p,q+1)}
    > 
    m^{(k)}_{(p,q)} \, m^{(k)}_{(p+1,q+1)}.
\]
Consequently,
\[
    h(p,q) = \frac{m^{(k)}_{(p+1,q)}}{m^{(k)}_{(p,q)}}
    > 
    \frac{m^{(k)}_{(p+1,q+1)}}{m^{(k)}_{(p,q+1)}} = h(p,q+1),
\]
and
\[
    v(p,q) = \frac{m^{(k)}_{(p,q+1)}}{m^{(k)}_{(p,q)}} 
    > 
    \frac{m^{(k)}_{(p+1,q+1)}}{m^{(k)}_{(p+1,q)}} = v(p+1,q).
\]
\end{lemma}

\begin{proof}
Let $A = (0,0)$, $B = (p+1,q+1)$, $C = (p,q+1)$, and $D = (0,1)$.  
These points satisfy condition~(i) (in case $q>0$) or (ii) (in case $q=0$) of Proposition~\ref{prop:ptolemy}.  
Applying the Ptolemy inequality yields
\[
    |AC|_k \cdot |BD|_k > |AB|_k \cdot |CD|_k.
\]

Using the identifications
\[
    |AC|_k = m^{(k)}_{(p,q+1)}, \quad 
    |BD|_k = m^{(k)}_{(p+1,q)}, \quad 
    |AB|_k = m^{(k)}_{(p+1,q+1)}, \quad 
    |CD|_k = m^{(k)}_{(p,q)},
\]
we obtain the desired inequality
\[
    m^{(k)}_{(p+1,q)} \, m^{(k)}_{(p,q+1)}
    > 
    m^{(k)}_{(p,q)} \, m^{(k)}_{(p+1,q+1)}.
    \]
Dividing both sides appropriately yields the stated comparisons for $h(p,q)$ and $v(p,q)$.
\end{proof}

\begin{cor}\label{cor:hv_bounds}
For every $(p,q) \in \mathcal{R}$, the following bounds hold:
\begin{equation}\label{eq:hv}
    h(q,q) \leq h(p,q) \leq h(p,1), \qquad 
    v(p,0) \leq v(p,q) \leq v(p,p-1).
\end{equation}
In particular, for all $(p,q) \in \mathcal{R}$, $h(p,q) > 1$ and $v(p,q) > 1$. 
\end{cor}

\begin{proof}
The inequalities in \eqref{eq:hv} follow by iteratively applying Lemmas~\ref{lem:hv1} and~\ref{lem:hv2}. Specifically, we do the following.

\begin{itemize}
    \item Fixing $q$, we can use Lemma~\ref{lem:hv1} to show $h(p,q) < h(p-1,q) < \cdots < h(q,q)$. Similarly fixing $p$ we can use Lemma~\ref{lem:hv2} to show $h(p,q) > h(p,q-1)  > \cdots > h(p,1)$. 
    \item A similar argument applies to $v(p,q)$, yielding $v(p,0) \leq v(p,q) \leq v(p,p-1)$.
\end{itemize}

To establish the strict lower bounds, we observe that by Lemma~\ref{prop:Intertwine} (see also \cite{banaian2025}), the boundary terms satisfy
\[
    h(q,q) = \frac{m^{(k)}_{(q+1,q)}}{m^{(k)}_{(q,q)}} > 1 
    \quad \text{and} \quad 
    v(p,0) = \frac{m^{(k)}_{(p,1)}}{m^{(k)}_{(p,0)}} > 1.
\] 
Combining these strict inequalities with the bounds in \eqref{eq:hv} immediately yields $h(p,q) > 1$ and $v(p,q) > 1$ for all $(p,q)\in \mathcal{R}$. 
\end{proof}

The following proposition is the main result of this subsection.

\begin{prop}\label{prop:monotonicity_hv}
Let $\ell$ be a horizontal line $y = b$ or a vertical line $x = n$, where $b,n \in \mathbb{Z}_{\geq 0}$.  
Then the generalized $k$-Markov numbers are strictly increasing along $\ell$ with respect to the enumeration of points in $\ell \cap \mathcal{R}$.
\end{prop}

\begin{proof}
The ratio of consecutive points on the line $y = b$, i.e.  $(p,b)$ and $(p+1,b)$, is $h(p,b)$, and by Corollary~\ref{cor:hv_bounds}, the sequence is strictly increasing. The same can be said for points along the line $x = n$, using the fact that $v(n,q) > 1$ for all $(n,q) \in \mathcal{R}$ with $q < n$.
\end{proof}

One can also use Lemma~\ref{lem:Slope0and1} to see that the sequence $m^{(k)}_{(p,0)}$ is increasing.
The following corollary is immediate from Proposition~\ref{prop:monotonicity_hv}. 

\begin{cor}\label{cor:positiveslope}
Let $\ell: y = ax + b$ be a line with rational slope and intercept, i.e., $a, b \in \mathbb{Q}$.  
If $a > 0$, then the generalized $k$-Markov numbers are strictly increasing as functions of the $x$-coordinate along $\ell \cap \mathcal{R}$.
\end{cor}

One can also use Lemma~\ref{lem:Slope0and1} to see that the sequence $m^{(k)}_{(p,p)}$ is increasing.

\subsection{Monotonicity along arbitrary lines}

In this section, we focus on characterizing all lines $\ell: y = ax + b$  with $a,b \in \mathbb{Q}$ along which the $k$-Markov numbers are monotonic.
By Proposition~\ref{prop:monotonicity_hv} (covering horizontal lines) and Corollary~\ref{cor:positiveslope} (which establishes strict increase for lines with positive slope), it suffices to consider the case $a < 0$.  
Moreover, we may assume $b > 0$, since otherwise $\ell \cap \mathcal{R} = \emptyset$.  

We can therefore write $a = -\frac{a_1}{a_2}$ for some coprime positive integers $a_1, a_2 \in \mathbb{Z}_{>0}$.  
Because $a < 0$, the intersection $\ell \cap \mathcal{R}$ is finite; denote its cardinality by $n(\ell) := |\ell \cap \mathcal{R}|.$

Throughout this subsection, we order the points of $\ell \cap \mathcal{R}$ using their $x$-coordinates. That is, set
\[
\mathbf{p}_1 = (p_1, q_1),\ 
\mathbf{p}_2 = (p_2, q_2),\ 
\dots,\ 
\mathbf{p}_{n(\ell)} = (p_{n(\ell)}, q_{n(\ell)}),
\]
so that $p_1 < p_2 < \cdots < p_{n(\ell)}$.

Since the slope $a$ is rational and $\mathcal{R} \subset \mathbb{Z}^2$, the set $\ell \cap \mathcal{R}$ forms an arithmetic progression in the integer lattice. Specifically, for any three consecutive points  
\[
\mathbf{p}_i = (p_i, q_i),\quad 
\mathbf{p}_{i+1} = (p_{i+1}, q_{i+1}),\quad 
\mathbf{p}_{i+2} = (p_{i+2}, q_{i+2}) ,
\]
with $1 \leq i \leq n(\ell) - 2$, we have 
\[
p_{i+1} = p_i + a_2,\quad q_{i+1} = q_i - a_1,
\qquad
p_{i+2} = p_{i+1} + a_2,\quad q_{i+2} = q_{i+1} - a_1.
\]

For each index $i$ with $1 \leq i < n(\ell)$, we define the \textbf{$i$-th ratio of the generalized $k$-Markov numbers along $\ell$} by
\[
\operatorname{Ratio}_i(\ell) := \frac{m^{(k)}_{\mathbf{p}_{i+1}}}{m^{(k)}_{\mathbf{p}_i}}.
\]
We will be particularly interested in the \textbf{first} and \textbf{last ratios}:
\begin{equation}\label{eq:extremal_ratios}
\operatorname{Ratio}_{\mathrm{first}}(\ell) := \operatorname{Ratio}_{1}(\ell),
\qquad
\operatorname{Ratio}_{\mathrm{last}}(\ell) := \operatorname{Ratio}_{\,n(\ell)-1}(\ell).
\end{equation}

\begin{prop}\label{prop:negative_slope}
Let $\ell: y = ax + b$ be a line with rational slope and intercept, where $a = -\dfrac{a_1}{a_2}$ for coprime positive integers $a_1, a_2 \in \mathbb{Z}_{>0}$.  
Then for every index $i$ with $1 \leq i \leq n(\ell) - 2$, we have
\[
\operatorname{Ratio}_i(\ell) < \operatorname{Ratio}_{i+1}(\ell).
\]
Consequently, the sequence of generalized $k$-Markov numbers along $\ell \cap \mathcal{R}$ is strictly increasing (respectively, strictly decreasing) as a function of the $x$-coordinate if and only if 
\[
\operatorname{Ratio}_{\mathrm{first}}(\ell) > 1 \quad \text{(respectively, } \operatorname{Ratio}_{\mathrm{last}}(\ell) < 1\text{)}.
\]
\end{prop}

\begin{proof}
Set $A = O = (0,0)$, $B = \mathbf{p}_{i+1}$, $C = \mathbf{p}_{i+2}$, and $D = (a_2,-a_1)$.
These four points form a convex quadrilateral satisfying condition~(a) of Proposition~\ref{prop:ptolemy}. Applying the Ptolemy inequality yields
\[
|AC|_k \cdot |BD|_k > |AB|_k \cdot |CD|_k.
\]

Using the identifications
\[
|AC|_k = m^{(k)}_{\mathbf{p}_{i+2}}, \quad 
|BD|_k = m^{(k)}_{\mathbf{p}_i}, \quad 
|AB|_k = |CD|_k = m^{(k)}_{\mathbf{p}_{i+1}},
\]
we obtain
\[
m^{(k)}_{\mathbf{p}_{i+2}} \, m^{(k)}_{\mathbf{p}_i} > \bigl(m^{(k)}_{\mathbf{p}_{i+1}}\bigr)^2.
\]
Dividing both sides by $m^{(k)}_{\mathbf{p}_i} \, m^{(k)}_{\mathbf{p}_{i+1}} > 0$, we deduce
\[
\frac{m^{(k)}_{\mathbf{p}_{i+2}}}{m^{(k)}_{\mathbf{p}_{i+1}}} 
> 
\frac{m^{(k)}_{\mathbf{p}_{i+1}}}{m^{(k)}_{\mathbf{p}_{i}}},
\]
which is precisely $\operatorname{Ratio}_{i+1}(\ell) > \operatorname{Ratio}_{i}(\ell)$, as claimed.
\end{proof}

\begin{lemma}\label{lem:ratio_comparison}
Let $\mathbf{p}_1 = (p_1, q_1)$, $\mathbf{p}_2 = (p_2, q_2)$, 
$\mathbf{p}'_1 = (p'_1, q'_1)$, and $\mathbf{p}'_2 = (p'_2, q'_2)$
be four points in $\mathcal{R}$ such that
\[
p_1 < p_2,\quad q_1 > q_2,\qquad 
p'_1 < p'_2,\quad q'_1 > q'_2,\qquad \text{and}\quad p_2 < p'_2.
\]
Assume further that $\mathbf{p}_1, \mathbf{p}_2, \mathbf{p}'_2, \mathbf{p}'_1$ form a parallelogram; equivalently,
$\mathbf{p}_2 - \mathbf{p}_1 = \mathbf{p}'_2 - \mathbf{p}'_1.$
Then the following hold:

\begin{enumerate}
    \item[$(1)$] If the segments $\overline{O\mathbf{p}'_2}$ and $\overline{O'\mathbf{p}_2}$ intersect in their interiors or at $\mathbf p_2$, where $O = (0,0)$ and $O' = (p_2 - p_1,\, q_2 - q_1)$ (for example, when the common slope of $\overline{\mathbf{p}_1\mathbf{p}'_1}$ and $\overline{\mathbf{p}_2\mathbf{p}'_2}$ is $0$), then
    \[
        \frac{m^{(k)}_{\mathbf{p}_2}}{m^{(k)}_{\mathbf{p}_1}} 
        < 
        \frac{m^{(k)}_{\mathbf{p}'_2}}{m^{(k)}_{\mathbf{p}'_1}}.
    \]

    \item[$(2)$]  If the segments $\overline{O\mathbf{p}_2}$ and $\overline{O'\mathbf{p}'_2}$ intersect in their interiors or at $\mathbf p_2$, where $O = (0,0)$ and $O' = (p_2 - p_1,\, q_2 - q_1)$ (for example, when the common slope of $\overline{\mathbf{p}_1\mathbf{p}'_1}$ and $\overline{\mathbf{p}_2\mathbf{p}'_2}$ is $1$), then
    \[
        \frac{m^{(k)}_{\mathbf{p}_2}}{m^{(k)}_{\mathbf{p}_1}} 
        >
        \frac{m^{(k)}_{\mathbf{p}'_2}}{m^{(k)}_{\mathbf{p}'_1}}.
    \]
\end{enumerate}
\end{lemma}

\begin{proof}
(1). Let $A=O=(0,0), B={\bf p}_2, C={\bf p}'_2$ and $D=O'=(p_2-p_1,q_2-q_1)$. These points satisfy condition~(a) or condition~(b) of Proposition~\ref{prop:ptolemy}.  
Applying the Ptolemy inequality yields
\[
    |AC|_k \cdot |BD|_k > |AB|_k \cdot |CD|_k.
\]

Using the identifications
\[
    |AC|_k = m^{(k)}_{{\bf p}'_2}, \quad 
    |BD|_k = m^{(k)}_{{\bf p}_1}, \quad 
    |AB|_k = m^{(k)}_{{\bf p}_2}, \quad 
    |CD|_k = m^{(k)}_{{\bf p}'_1},
\]
we obtain the desired inequality
\[
    m^{(k)}_{{\bf p}'_2} \, m^{(k)}_{{\bf p}_1} 
    > 
   m^{(k)}_{{\bf p}_2} \, m^{(k)}_{{\bf p}'_1},
\]
implying 
\[
    \frac{m^{(k)}_{{\bf p}_2}}{m^{(k)}_{{\bf p}_1}} 
    < 
    \frac{m^{(k)}_{{\bf p}'_2}}{m^{(k)}_{{\bf p}'_1}}.
    \]

(2). Let $A=O=(0,0), B=\mathbf{p}_2', C=\mathbf{p}_2, D=O'=(p_2-p_1,q_2-q_1)$. It follows that the points $A,B,C,D$ satisfy condition~(a) or condition~(b) of Proposition~\ref{prop:ptolemy}.  
Applying the Ptolemy inequality yields
\[
    |AC|_k \cdot |BD|_k > |AB|_k \cdot |CD|_k.
\]

Using the identifications
\[
    |AC|_k = m^{(k)}_{\mathbf{p}_2}, \quad 
    |BD|_k = m^{(k)}_{\mathbf{p}'_1}, \quad 
    |AB|_k = m^{(k)}_{\mathbf{p}'_2}, \quad 
    |CD|_k = m^{(k)}_{\mathbf{p}_1},
\]
we obtain the desired inequality
\[
    m^{(k)}_{{\bf p}_2} \, m^{(k)}_{{\bf p}'_1} 
    > 
   m^{(k)}_{{\bf p}'_2} \, m^{(k)}_{{\bf p}_1},
\]
implying
\[
    \frac{m^{(k)}_{{\bf p}_2}}{m^{(k)}_{{\bf p}_1}} 
    >
    \frac{m^{(k)}_{{\bf p}'_2}}{m^{(k)}_{{\bf p}'_1}}.
    \]
\end{proof}

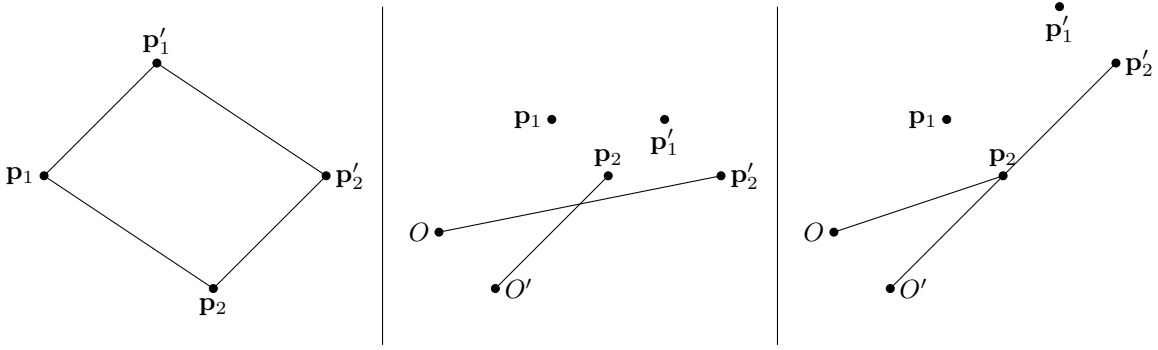
\begin{figure}
    \centering
\begin{tikzpicture}[scale=0.75]
\draw[fill = black] (1,1) circle [radius=2pt];
\node[ left] at (1,1){${\bf p}_1$};
\draw[fill = black] (4,-1) circle [radius=2pt];
\node[below] at (4,-1){${\bf p}_2$};
\draw[fill = black] (3,3) circle [radius=2pt];
\node[above] at (3,3){${\bf p}'_1$};
\draw[fill = black] (6,1) circle [radius=2pt];
\node[right] at (6,1){${\bf p}'_2$};
\draw(1,1) -- (4,-1)--(6,1) -- (3,3) -- (1,1); 
\draw(7,4) -- (7,-2);
\draw[fill = black] (10,2) circle [radius=2pt];
\node[ left] at (10,2){${\bf p}_1$};
\draw[fill = black] (12,2) circle [radius=2pt];
\node[below] at (12,2){${\bf p}'_1$};
\draw[fill = black] (11,1) circle [radius=2pt];
\node[above] at (11,1){${\bf p}_2$};
\draw[fill = black] (13,1) circle [radius=2pt];
\node[right] at (13,1){${\bf p}'_2$};
\draw[fill = black] (8,0) circle [radius=2pt];
\node[left] at (8,0){$O$};
\draw[fill = black] (9,-1) circle [radius=2pt];
\node[right] at (9,-1){$O'$};
\draw(8,0) -- (13,1);
\draw(9,-1) -- (11,1);
\draw(14,4) -- (14,-2);
\draw[fill = black] (17,2) circle [radius=2pt];
\node[ left] at (17,2){${\bf p}_1$};
\draw[fill = black] (19,4) circle [radius=2pt];
\node[below] at (19,4){${\bf p}'_1$};
\draw[fill = black] (18,1) circle [radius=2pt];
\node[above] at (18,1){${\bf p}_2$};
\draw[fill = black] (20,3) circle [radius=2pt];
\node[right] at (20,3){${\bf p}'_2$};
\draw[fill = black] (15,0) circle [radius=2pt];
\node[left] at (15,0){$O$};
\draw[fill = black] (16,-1) circle [radius=2pt];
\node[right] at (16,-1){$O'$};
\draw(15,0) -- (18,1);
\draw(16,-1) -- (20,3);
\end{tikzpicture}
    \caption{Left: A generic example of points ${\bf p}_1, \bf{p}'_1,\bf{p}_2,$ and $\bf{p}'_2$ as in Lemma ~\ref{lem:ratio_comparison}; Middle: Four points such that $\overline{O{\bf p}'_2}$ and $\overline{O'\bf{p}_2}$ intersect; Right: Four points such that $\overline{O{\bf p}_2}$ and $\overline{O'\bf{p}'_2}$ intersect}
    \label{fig:FiguresForComparisonLemma}
\end{figure}

Given two coprime positive integers $a_1$ and $a_2$, and $n \geq 1+a_1+a_2$, let ${\bf p}_1=(n-a_2,1+a_1), {\bf p}_2=(n,1), {\bf p}'_1=(n+1-a_2,1+a_1)$ and ${\bf p}'_2=(n+1,1)$. The line segments $\overline{O{\bf p}'_2}$ and $\overline{O'{\bf p}_2} $ intersect in their interiors, where $O = (0,0)$ and $O' = (a_2,\,-a_1)$.  
By Lemma~\ref{lem:ratio_comparison} (1), it follows that the sequence
$$\left\{ \frac{m^{(k)}_{(n,\,1)}}{m^{(k)}_{(n - a_2,\,1 + a_1)}} \;\middle|\; n \geq 1 + a_1 + a_2 \right\}
$$
is strictly increasing. In particular, this sequence contains the subsequence
$$
\left\{ \frac{m^{(k)}_{(q(1 + a_1) + a_2,\,1)}}{m^{(k)}_{(q(1 + a_1),\,1 + a_1)}} \;\middle|\; q \in \mathbb{Z}_{>0} \right\}.
$$

Recall the definition of $S_\pm(a_1,a_2)$ from~ \eqref{eq:a_pm_def}. Thus, we see thus

\begin{equation}\label{eq:limits-}
\lim_{n\to \infty}\frac{m^{(k)}_{\bigl(n,\; 1\bigr)}}{m^{(k)}_{\bigl(n-a_2,\; 1+a_1\bigr)}}=S(a_1,a_2)_-.
\end{equation}

Similarly, by setting  ${\bf p}_1=(n+1,n), {\bf p}_2=(n+1+a_2,n-a_1), {\bf p}'_1=(n+2,n+1)$ and ${\bf p}'_2=(n+1+a_2,n+1-a_1)$ and using Lemma~\ref{lem:ratio_comparison} (2), we have that the sequence
 
$$\left\{ \frac{m^{(k)}_{(n+1+a_2,\,n-a_1)}}{m^{(k)}_{(n+1,\,n)}} \;\middle|\; n \geq a_1 \right\}
$$
is strictly increasing. In particular, it contains the subsequence
$$
\left\{  \frac{m^{(k)}_{\bigl((q+1)(1+a_1+a_2),\; q(1+a_1+a_2)\bigr)}}{m^{(k)}_{\bigl(q(1+a_1+a_2)+a_1+1,\; q(1+a_1+a_2)+a_1\bigr)}} \;\middle|\; q \in \mathbb{Z}_{>0} \right\}.
$$

so we again have that the whole sequence converges, i.e., 

\begin{equation}\label{eq:limits+}
\lim_{n\to \infty}\frac{m^{(k)}_{(n+1+a_2,\,n-a_1)}}{m^{(k)}_{(n+1,\,n)}} =S(a_1,a_2)_+.
\end{equation}

Recall $\ell: y = ax + b$  denotes a line with rational slope and intercept, where  $ a = -\dfrac{a_1}{a_2} $  for coprime positive integers  $ a_1, a_2 \in \mathbb{Z}_{>0} $.  For any  $ n \in \mathbb{Z}_{>0} $ , define the following two shifted lines:
\begin{equation*}
    \ell[n]:\ y = a(x - n) + b, \qquad 
    \ell\langle n\rangle:\ y - n = a(x - n) + b.
\end{equation*}

Suppose that $n(\ell)=|\ell\cap \mathcal R|\geq 2$. With our indexing of the points on $\ell \cap \mathcal{R}$,  $\mathbf{p}_1 = (p_1, q_1) $  and  $\mathbf{p}_2=(p_2, q_2)$  are the leftmost two points of  $\ell \cap \mathcal{R}$  (i.e., those with smallest  $ x $ -coordinates). For convenience, let $ \widetilde{\mathbf{p}}_1 = (\widetilde{p}_1, \widetilde{q}_1) $  and  $ \widetilde{\mathbf{p}}_2 = (\widetilde{p}_2, \widetilde{q}_2) $  be the rightmost two points of  $ \ell \cap \mathcal{R} $  (i.e., those with largest  $ x $ -coordinates). That is, $\widetilde{\mathbf{p}}_1 = \mathbf{p}_{n(\ell) - 1}$ and  $\widetilde{\mathbf{p}}_2 = \mathbf{p}_{n(\ell)}$ .

Since the lines $\ell[n]$  and $\ell \langle n \rangle $ are the result of translating $\ell$ in a direction parallel to one of the boundary lines of $\mathcal{R}$, we have the following.
\begin{itemize}
   \item The points 
    $\widetilde{\mathbf{p}}_1[n] := (\widetilde{p}_1 + n,\, \widetilde{q}_1), \quad 
    \widetilde{\mathbf{p}}_2[n] := (\widetilde{p}_2 + n,\, \widetilde{q}_2)$ are the rightmost two points of  $ \ell[n] \cap \mathcal{R} $.
    \item The points 
    $\mathbf{p}_1\langle n\rangle := (p_1 + n,\, q_1 + n),  \quad
    \mathbf{p}_2\langle n\rangle := (p_2 + n,\, q_2 + n)$
are the leftmost two points of  $ \ell\langle n\rangle \cap \mathcal{R}$.
\end{itemize}

\begin{prop}\label{prop:firstlastratios}
Let $\ell:y=ax+b$ be a line with rational slope and intercept, where  $a=-\dfrac{a_1}{a_2}$ for coprime positive integers  $a_1,a_2 \in \mathbb{Z}_{>0}$. Suppose that $n(\ell)=|\ell\cap \mathcal R|\geq 2$. 
\begin{enumerate}
    \item[$(1)$] The sequence $\{\operatorname{Ratio}_{\mathrm{last}}(\ell[ n])\mid n\geq 0\}$ is strictly increasing and 
    \begin{equation}\label{eq:ratiolim1}
      \lim_{n\to \infty} \operatorname{Ratio}_{\mathrm{last}}(\ell[n])=S_-(a_1,a_2). 
    \end{equation}
     \item[$(2)$] The sequence $\{\operatorname{Ratio}_{\mathrm{first}}(\ell\langle n\rangle)\mid n\geq 0\}$ is strictly decreasing and 
    \begin{equation}\label{eq:ratiolim2}
      \lim_{n\to \infty} \operatorname{Ratio}_{\mathrm{first}}(\ell\langle n\rangle)=S_+(a_1,a_2). 
    \end{equation}
\end{enumerate}
  
\end{prop}

\begin{proof}
We begin with part (1). The fact that the sequence $\{\operatorname{Ratio}_{\mathrm{last}}(\ell[n]) \mid n \geq 0\}$ is strictly increasing follows immediately from Lemma~\ref{lem:ratio_comparison} (1). In particular, set $\mathbf{p}_1 = \widetilde{\mathbf{p}_1}[n],\mathbf{p}_2 = \widetilde{\mathbf{p}_2}[n], \mathbf{p}_1' = \widetilde{\mathbf{p}_1}[n+1]$, and $\mathbf{p}_2' = \widetilde{\mathbf{p}_2}[n+1]$.

    We now prove~\eqref{eq:ratiolim1}. 
    If $\widetilde{q}_2 = 1$, then~\eqref{eq:ratiolim1} coincides with~\eqref{eq:limits-}. 
    Suppose instead that $\widetilde{q}_2 > 1$. 
    Then $\frac{b-1}{a}\notin \mathbb Z$. 
    Let $N(\ell)$ be the unique positive integer satisfying 
    \[
        N(\ell) - 1 < -\frac{b-1}{a} < N(\ell),
    \]
    and define the points
    \[
        A = (N(\ell) - 1 - a_2,\; 1 + a_1), \quad B = (N(\ell) - 1,\; 1),
    \]
    \[
        A' = (N(\ell) - a_2,\; 1 + a_1), \quad B' = (N(\ell),\; 1).
    \]
Therefore, the lines $\ell$, $\overline{AB}$, and $\overline{A'B'}$ share the same slope, with $\ell$ lying between $\overline{AB}$ and $\overline{A'B'}$. We thus have $\widetilde p_2<N(\ell)$. We also have $1<\widetilde q_2 < 1 + a_1$; otherwise, the point $(\widetilde p_2 + a_2, q_2 - a_1)$ would belong to $\ell \cap \mathcal R$, contradicting the fact that $\widetilde{\mathbf{p}}_2$ is the rightmost point of $\ell \cap \mathcal R$.

Hence, $\widetilde{\mathbf{p}}_2$ lies in the interior of the parallelogram with vertices $A$, $A'$, $B$, and $B'$. As a result, the segments $\overline{OB'}$ and $\overline{O'\widetilde{\mathbf{p}}_2}$ intersect in their interiors, where $O = (0,0)$ and $O' = (a_2, -a_1)$. 

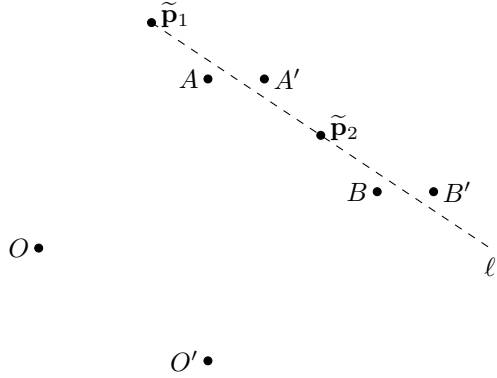
\begin{center}
\begin{tikzpicture}[scale = 0.75]
\draw[fill = black] (0,0) circle [radius=2pt];
\node[left] at (0,0){$O$};
\draw[fill = black] (3,-2) circle [radius=2pt];
\node[left] at (3,-2){$O'$};
\draw[fill = black] (6,1) circle [radius=2pt];
\node[left] at (6,1){$B$};
\draw[fill = black] (7,1) circle [radius=2pt];
\node[right] at (7,1){$B'$};
\draw[fill = black] (3,3) circle [radius=2pt];
\node[left] at (3,3){$A$};
\draw[fill = black] (4,3) circle [radius=2pt];
\node[right] at (4,3){$A'$};
\draw[fill = black] (5,2) circle [radius=2pt];
\node[right, yshift = 3pt] at (5,2){$\widetilde{\mathbf{p}}_2$};
\draw[fill = black] (2,4) circle [radius=2pt];
\node[right, yshift = 3pt] at (2,4){$\widetilde{\mathbf{p}}_1$};
\draw[dashed] (2,4) -- (8,0);
\node[below] at (8,0){$\ell$};
\end{tikzpicture}
\end{center}

 Since \( n(\ell) = |\ell \cap \mathcal R| \ge 2 \), we have \( \widetilde{\bf p}_1 = (\widetilde p_1, \widetilde q_1) \in \mathcal R \), and consequently  
\(
\widetilde p_1 = \widetilde p_2 - a_2 \geq \widetilde q_1 = \widetilde q_2 + a_1.
\)
Using \( \widetilde p_2 < N(\ell) \) and \( \widetilde q_2 > 1 \), we obtain  
\(
N(\ell) - a_2 > 1 + a_1.
\)
Hence, \( A', B' \in \mathcal R \). Moreover, using $p_2 = p_1 + a_2$ and $q_2 = q_1 - a_1$, we can see that the slopes of $\overline{{\bf p}_1 A'}$ and $\overline{{\bf p}_2 B'}$ agree. Therefore, by Lemma~\ref{lem:ratio_comparison}(1), we obtain
    \begin{equation*}
        \operatorname{Ratio}_{\mathrm{last}}(\ell)
        = \frac{m^{(k)}_{\widetilde{\mathbf{p}}_2}}{m^{(k)}_{\widetilde{\mathbf{p}}_1}}
        < \frac{m^{(k)}_{B'}}{m^{(k)}_{A'}}
        = \frac{m^{(k)}_{(N(\ell),\, 1)}}{m^{(k)}_{(N(\ell) - a_2,\, 1 + a_1)}}.
    \end{equation*}
    Translating everything except $O$ and $O'$ by $n$ units 
    horizontally, we get for all $n \geq 0$,
    \begin{equation*}
        \operatorname{Ratio}_{\mathrm{last}}(\ell[n])
        = \frac{m^{(k)}_{\widetilde{\mathbf{p}}_2[n]}}{m^{(k)}_{\widetilde{\mathbf{p}}_1[n]}}
        < \frac{m^{(k)}_{B'[n]}}{m^{(k)}_{A'[n]}}
        = \frac{m^{(k)}_{(N(\ell) + n,\, 1)}}{m^{(k)}_{(N(\ell) + n - a_2,\, 1 + a_1)}}.
    \end{equation*}

    Hence, the sequence $\{\operatorname{Ratio}_{\mathrm{last}}(\ell[n])\}$ is increasing and bounded above by a convergent sequence, and therefore the limit $\lim_{n \to \infty} \operatorname{Ratio}_{\mathrm{last}}(\ell[n])$ exists; denote it by $c_-$. 
    Moreover, from (\ref{eq:limits-}), we have $c_- \leq S(a_1, a_2)_-$.

    Conversely, fix any $N > 1+a_1+a_2$, and set $A = (N - a_2,\, 1 + a_1)$ and $B = (N,\, 1)$. 
    Recall we set $O' = (a_2,-a_1)$. Since the slope of $\overline{O'B}$ is positive, there exists $n > 0$ such that the segments $\overline{O'B}$  and $\overline{O\widetilde{\mathbf{p}}_2[n]}$ intersect in their interiors. 

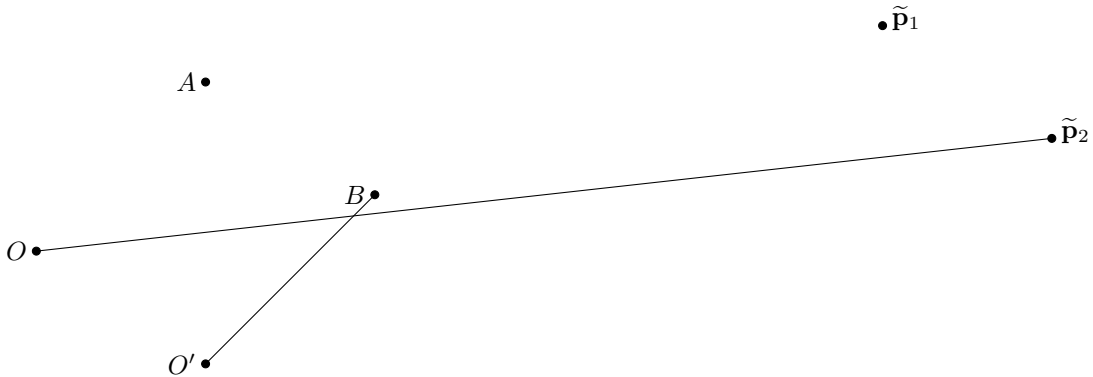
\begin{center}
\begin{tikzpicture}[scale = 0.75]
\draw[fill = black] (0,0) circle [radius=2pt];
\node[left] at (0,0){$O$};
\draw[fill = black] (3,-2) circle [radius=2pt];
\node[left] at (3,-2){$O'$};
\draw[fill = black] (6,1) circle [radius=2pt];
\node[left] at (6,1){$B$};
\draw[fill = black] (18,2) circle [radius=2pt];
\node[right, yshift = 3pt] at (18,2){$\widetilde{\mathbf{p}}_2$};
\draw[fill = black] (3,3) circle [radius=2pt];
\node[left] at (3,3){$A$};
\draw[fill = black] (15,4) circle [radius=2pt];
\node[right, yshift = 3pt] at (15,4){$\widetilde{\mathbf{p}}_1$};
\draw(0,0) -- (18,2);
\draw(3,-2) -- (6,1);
\end{tikzpicture}
\end{center}

    Applying Lemma~\ref{lem:ratio_comparison}(1) again gives
    \begin{equation*}
        \frac{m^{(k)}_{(N,\, 1)}}{m^{(k)}_{(N - a_2,\, 1 + a_1)}}
        = \frac{m^{(k)}_B}{m^{(k)}_A}
        < \frac{m^{(k)}_{\widetilde{\mathbf{p}}_2[n]}}{m^{(k)}_{\widetilde{\mathbf{p}}_1[n]}}
        = \operatorname{Ratio}_{\mathrm{last}}(\ell[n]).
    \end{equation*}
    Taking the limit inferior as $n \to \infty$, we deduce $S(a_1, a_2)_- \leq c_-$.

    Combining both inequalities, we conclude that
    \[
        \lim_{n \to \infty} \operatorname{Ratio}_{\mathrm{last}}(\ell[n]) = c_- = S(a_1, a_2)_-.
    \]

Next we consider part (2). Showing that the sequence $\{\operatorname{Ratio}_{\mathrm{first}}(\ell\langle n\rangle) \mid n \geq 0\}$ is strictly decreasing again follows directly from Lemma~\ref{lem:ratio_comparison}(2).  In particular, set $\mathbf{p}_1 = \mathbf{p}_1\langle n\rangle,\mathbf{p}_2 = \mathbf{p}_2\langle n\rangle, \mathbf{p}_1' = \mathbf{p}_1\langle n+1 \rangle$, and $\mathbf{p}_2' = \mathbf{p}_2\langle n+1\rangle$. 

We now turn to the proof of~\eqref{eq:ratiolim2}.  
If ${\mathbf p}_1 = (n+1,n)$ for some $n$, then~\eqref{eq:ratiolim2} reduces to~\eqref{eq:limits+}.  
Suppose instead that ${\mathbf p}_1 \neq (n+1,n)$ for any $n$; in this case, $\frac{b+1}{1-a}\notin \mathbb Z$. Let $N(\ell)$ be the unique positive integer such that
\[
N(\ell) - 1 < \frac{b+1}{1-a} < N(\ell),
\]
and define the points
\[
A = (N(\ell), \, N(\ell)-1),\quad B = (N(\ell) + a_2,\; N(\ell)-1-a_1),
\]
\[
A' = (N(\ell)+1, \; N(\ell)),\quad B' = (N(\ell) +1 + a_2,\; N(\ell)-a_1).
\]
Thus, the lines $\ell$, $\overline{AB}$, and $\overline{A'B'}$ all have the same slope, with $\ell$ lying between $\overline{AB}$ and $\overline{A'B'}$. Consequently, the point ${\mathbf{p}}_1$ lies inside the parallelogram with vertices $A, A', B, B'$. Moreover, because $\overline{O'B'}$ crosses $\overline{AB}$ and
$\overrightarrow{AB} = \overrightarrow{\mathbf p_1 \mathbf p_2}$
, the segments $\overline{O{\mathbf{p}}_2}$ and $\overline{O'B'}$ intersect in their interiors, where $O = (0,0)$ and $O' = (a_2, -a_1)$. As in (1), since $n(\ell) = |\ell \cap \mathcal R| \geq 2$, we have $A', B' \in \mathcal R$.

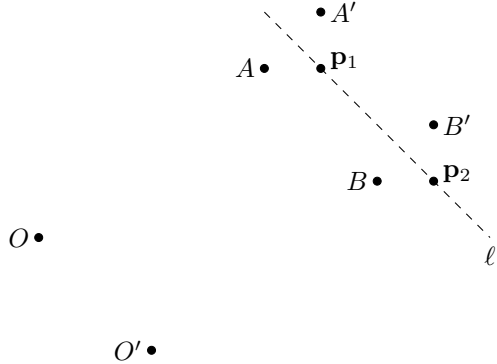
\begin{center}
\begin{tikzpicture}[scale = 0.75]
\draw[fill = black] (0,0) circle [radius=2pt];
\node[left] at (0,0){$O$};
\draw[fill = black] (2,-2) circle [radius=2pt];
\node[left] at (2,-2){$O'$};
\draw[fill = black] (6,1) circle [radius=2pt];
\node[left] at (6,1){$B$};
\draw[fill = black] (7,2) circle [radius=2pt];
\node[right] at (7,2){$B'$};
\draw[fill = black] (4,3) circle [radius=2pt];
\node[left] at (4,3){$A$};
\draw[fill = black] (5,4) circle [radius=2pt];
\node[right] at (5,4){$A'$};
\draw[fill = black] (5,3) circle [radius=2pt];
\node[right, yshift = 3pt] at (5,3){$\mathbf{p}_1$};
\draw[fill = black] (7,1) circle [radius=2pt];
\node[right, yshift = 3pt] at (7,1){$\mathbf{p}_2$};
\draw[dashed] (4,4) -- (8,0);
\node[below] at (8,0){$\ell$};
\end{tikzpicture}
\end{center}

Applying Lemma~\ref{lem:ratio_comparison}(2), we obtain
\begin{equation*}
\operatorname{Ratio}_{\mathrm{first}}(\ell)
= \frac{m^{(k)}_{{\mathbf{p}}_2}}{m^{(k)}_{{\mathbf{p}}_1}}
> \frac{m^{(k)}_{B'}}{m^{(k)}_{A'}}
= \frac{m^{(k)}_{(N(\ell) +1 + a_2,\; N(\ell)-a_1)}}{m^{(k)}_{(N(\ell)+1,\; N(\ell))}}.
\end{equation*}
Translating by $n$ units up and to the right yields, for all $n \geq 0$,
\begin{equation*}
\operatorname{Ratio}_{\mathrm{first}}(\ell\langle n\rangle)
= \frac{m^{(k)}_{{\mathbf{p}}_2\langle n\rangle}}{m^{(k)}_{{\mathbf{p}}_1\langle n\rangle}}
> \frac{m^{(k)}_{B'\langle n\rangle}}{m^{(k)}_{A'\langle n\rangle}}
= \frac{m^{(k)}_{(N(\ell) +1+n + a_2,\; N(\ell)+n - a_1)}}{m^{(k)}_{(N(\ell)+n+1,\; N(\ell)+n)}}.
\end{equation*}

    Hence, the limit $\lim_{n \to \infty} \operatorname{Ratio}_{\mathrm{first}}(\ell\langle n\rangle)$ exists; denote it by $c_+$. 
    Moreover, passing to the limit yields $c_+ \geq S(a_1, a_2)_+$.

    Conversely, fix any $N > a_1$, and set $A = (N+1,\, N)$ and $B = (N+1+a_2,\, N-a_1)$. 
    Since the slope of $\overline{OB}$ is positive, there exists $n > 0$ such that the segments $\overline{OB}$ and $\overline{O'{\mathbf{p}}_2\langle n\rangle}$ intersect in their interiors. 
    Applying Lemma~\ref{lem:ratio_comparison}(2) again gives
    \begin{equation*}
        \frac{m^{(k)}_{(N+1+a_2,\, N-a_1)}}{m^{(k)}_{(N+1,\, N)}}
        = \frac{m^{(k)}_B}{m^{(k)}_A}
        > \frac{m^{(k)}_{{\mathbf{p}}_2\langle n\rangle}}{m^{(k)}_{{\mathbf{p}}_1\langle n\rangle}}
        = \operatorname{Ratio}_{\mathrm{first}}(\ell\langle n\rangle).
    \end{equation*}
    Taking the limit inferior as $n \to \infty$, we deduce $S(a_1, a_2)_+ \geq c_-$.

    Combining both inequalities, we conclude that
    \[
       \lim_{n\to \infty} \operatorname{Ratio}_{\mathrm{first}}(\ell\langle n\rangle)=S_+(a_1,a_2). 
    \]
    The proof is complete.
\end{proof}

We now have collected all the tools needed for our main result, concerning the behavior of generalized $k$-Markov numbers along lines of various slopes.

\begin{theorem}\label{thm:Main}
Let $\ell:y=ax+b$ be a line with rational slope and intercept. 
Let $\alpha = 4k^2 + 12k + 5$, $\beta= 3k^2 + 8k + 6$ and $\delta= 3k^4 + 17k^3 + 34 k^2 + 28k + 8$ as in Proposition ~\ref{prop:kMarkovFibAndPell} and let $A= \frac{2k^2 + 3k + 1 + (1+k)\sqrt{\alpha}}{2(1+2k)\sqrt{\alpha}}$, $B=\frac{(k+1)(k+2)(\beta^2-4)+\delta\sqrt{\beta^2-4}}{(\beta-2)(\beta^2-4)}$. Then the following hold:
\begin{enumerate}
    \item[$(1)$] If $S_+(a_1,a_2)\geq 1$, or equivalently $a \geq -\frac{\ln(3+3k)B}{\ln\frac{\beta+\sqrt{\beta^2-4}}{2(3+3k)B}}$
    then the generalized $k$-Markov numbers are strictly increasing as functions of $x$ along $\ell\cap \mathcal R$. 

    \item[$(2)$] If $S_-(a_1,a_2)\leq 1$, or equivalently $a\leq -\frac{\ln\frac{3+2k+\sqrt{\alpha}}{2}}{\ln(3+3k)A}$,
    then the generalized $k$-Markov numbers are strictly decreasing as functions of $x$ along $\ell\cap \mathcal R$. 

    \item[$(3)$] If $S_+(a_1,a_2)<1$ and $S_-(a_1,a_2)>1$, or equivalently $-\frac{\ln\frac{3+2k+\sqrt{\alpha}}{2}}{\ln(3+3k)A} < a < \frac{\ln(3+3k)B}{\ln\frac{\beta+\sqrt{\beta^2-4}}{2(3+3k)B}}$, then for any 
    $b\in \mathbb Q$ such that $|\ell\cap \mathcal R|>2$, then
    the generalized $k$-Markov numbers are neither monotone increasing nor monotone decreasing in $x$ along $\ell\cap \mathcal R$. More precisely, they first decrease and then increase as $x$ grows.
\end{enumerate}
\end{theorem}
    
\begin{proof}
If $a\geq 0$, then the result is a direct consequence of Proposition ~\ref{prop:monotonicity_hv} and Corollary~\ref{cor:positiveslope}.

If  $a<0$, then the assertion follows immediately from Propositions~\ref{prop:negative_slope} and~\ref{prop:firstlastratios}.
\end{proof}

\subsection{Concluding Discussion}

Consider the functions $U(k):= -\frac{\ln(3+3k)B}{\ln\frac{\beta+\sqrt{\beta^2-4}}{2(3+3k)B}}$ and $L(k) =  -\frac{\ln\frac{3+2k+\sqrt{\alpha}}{2}}{\ln(3+3k)A}$, with $A,B,\alpha,\beta,$ and $\delta$ as defined previously. Let $(p,q), (s,t) \in \mathcal{R}$ be such that $p < s$. Let $\ell: ax + b$ be the line on which these two points lie. Theorem~\ref{thm:Main} tells us that, if $a \geq U(k)$, then $m^{(k)}_{(p,q)} < m^{(k)}_{(s,t)}$ and if $a \leq L(k)$, then $m^{(k)}_{(p,q)} < m^{(k)}_{(s,t)}$. We provide a few representative values of $U(k)$ and $L(k)$ below.

\begin{center}
\begin{tabular}{|c|c|c|c|c|c|c|c|}\hline
 $k$    & 0 & 1 & 2 & 3 & 100 & 1000 & 10,000 \\\hline
 $U(k)$   & -1.14320 & -1.14623 & -1.14320 & -1.13485 & -1.05725 & -1.03933 & -1.029917\\\hline
 $L(k)$ & -1.24167 & -1.30922 & -1.24167& -1.20123& -1.05827 & -1.03940 & -1.029923\\ \hline
\end{tabular}
\end{center}

This data reflects general properties of these functions. Both functions are increasing for $k \geq 1$ and both converge to $-1$. Therefore, the ``constant sum'' ordering is a sort of extreme instance of a monotonically increasing line. The fact that the two functions converge to the same value implies the difference between their values, $U(k) - L(k)$, converges to 0. Therefore, as $k$ increases past $1$, the $k$-Markov numbers are more likely to grow monotonically along an arbitrary line. 

This behavior makes a $k$-version of the Uniqueness Conjecture even more likely to be true than the ordinary version. For example, in \cite[Corollary 1.5]{LLRS}, Lee, Li, Rabideau, and Schiffler give an upper bound for the number of generalized Markov numbers $m_{(p',q')}$ which could be equal to a fixed number $m_{(p,q)}$. This is a function of $p$ and $q$ and comes from estimating the number of integral points in $\mathcal{R}$ which lie in a wedge with vertex $(p,q)$ and bounded by lines of slopes $-\frac54$ and $-\frac87$, close approximations to $L(0)$ and $U(0)$. Therefore, for any fixed Markov number, one can compute its uniqueness with a finite calculation. 

We can repeat these calculations for general $L(k)$ and $U(k)$, and this possible number of integral points will decrease. For example, from \cite[Corollary 1.5]{LLRS}, there are at most 832 Markov numbers equal to $m_{\frac{1009}{9973}}$, but the number of $10000$-Markov numbers which could be equal to  $m^{(10000)}_{\frac{1009}{9973}}$ is at most 23. 

Recall one way to prove the Uniqueness Conjecture would be to show that the partial order $(\mathbb{Q} \cap [0,1],\prec)$, given by $\fqp \prec \frac{q'}{p'}$ whenever $m_{\fqp} < m_{\frac{q'}{p'}}$, is a total order. Showing the generalized conjecture in \cite{LLRS} would entail exhibiting such a total order on all points in $\mathcal{R}$. 

Gyoda and Maruyama conjectured that the Uniqueness Conjecture remains true for $k$-Markov numbers in \cite[Conjecture 1.8] {gyoda2023uniqueness}. Accordingly, one can consider the partial order  $(\mathbb{Q} \cap [0,1],\prec^{(k)})$ given by $\fqp \prec^{(k)} \frac{q'}{p'}$ whenever $m^{(k)}_{\fqp} < m^{(k)}_{\frac{q'}{p'}}$. The first author asked whether the partial orders $\prec^{(k)}$ were the same for varying $k$ in \cite[Question 52]{banaian2025}. For instance, by \cite[Theorem 2.9]{gyoda2023uniqueness},  $m^{(2)}_{\fqp} = (m^{(0)}_{\fqp})^2$, so it is immediate that $\prec^{(0)}$ and $\prec^{(2)}$ coincide. However, the results provided here show that the answer to this Question in general is no. We end with a few examples illustrating some differences. 

\begin{example}
An interesting phenomenon for the functions $U(k)$ and $L(k)$ is that, when considering positive integral inputs, each is minimized at $k = 1$.Indeed, we see that the gray area $U(k)-L(k)$ is largest for $k =1$, so this is a natural place to search for examples of how the orderings can differ. 

For example, consider the two pairs of coprime numbers $(29,6)$ and $(25,11)$. The slope connecting these two points is $-\frac54$. Since $-\frac54 < L(0)$, we know $m^{(0)}_{\frac{11}{25}} > m^{(0)}_{\frac{6}{29}}$. (This inequality was already known by \cite[Theorem 1.2]{LLRS}.) Indeed, one can calculate \[
m^{(0)}_{\frac{11}{25}}  =  48795987025021 
\]
and
\[
m^{(0)}_{\frac{6}{29}}  =  46127828641049 
\]

However, $L(1) < -\frac54 < U(1)$, so we do not have the same guarantee for the associated 1-Markov numbers. And when we compute these, we find $m^{(1)}_{\frac{11}{25}} < m^{(1)}_{\frac{6}{29}}$:\[
m^{(1)}_{\frac{11}{25}}  =  9998020960587781820161
\]
and
\[
m^{(1)}_{\frac{6}{29}}  =  11854846326279367099921 .
\]
\end{example}

\begin{example}
For another example, consider the coprime pairs $(8,7)$ and $(13,1)$.  The slope of the line connecting these two points is $-\frac65$, which lies in the the gray zone for the first few values of $k$. The numbers $m^{(0)}_{\frac78}$ and $m^{(0)}_{\frac{1}{13}}$ are Pell and Fibonacci numbers respectively and one can calculate $m^{(0)}_{\frac78} = 195025$ and $m^{(0)}_{\frac{1}{13}} = 196418$. In particular, $m^{(0)}_{\frac78} < m^{(0)}_{\frac{1}{13}}$. Using Proposition~\ref{prop:kMarkovFibAndPell}, one can check again $m^{(1)}_{\frac78} < m^{(1)}_{\frac{1}{13}}$ and $m^{(2)}_{\frac78} < m^{(2)}_{\frac{1}{13}}$. The latter is indeed immediate from \cite[Theorem 2.9]{gyoda2023uniqueness}. However, at $k = 3$, the behavior swaps. We have \[
m^{(3)}_{\frac78} = 1188752792899
\]
and
\[
m^{(3)}_{\frac{1}{13}} = 1108609632005
\]
so that $m^{(2)}_{\frac78} > m^{(2)}_{\frac{1}{13}}$.
\end{example}

Based on this data and these investigations, we conclude with some items to consider in the future. We begin by conjecturing a strengthening of \cite[Conjecture 1.8] {gyoda2023uniqueness}.

\begin{conjecture}
 For all $(p,q),(p',q') \in \mathcal{R}$ and for all $k \geq 0$, $m^{(k)}_{(p,q)} = m^{(k)}_{(p',q')}$ implies $(p,q)=(p',q')$.   
\end{conjecture}

Next, based on the behavior of $U(k)$ and $L(k)$ as $k$ grows, we ask about the following relationship amongst the behavior of the orderings $\prec^{(k)}$ in their respective ``gray zones.''

\begin{quest}\label{quest:OrderingRelationship}
Do there exist $(p,q),(p',q') \in \mathcal{R}$ with $p<p'$ and $1 < k_1 < k_2$ such that (1) $L(k_i) < \frac{q'-q}{p'-p} < U(k_i)$ for $i = 1,2$, (2) $m^{(k_1)}_{(p,q)} > m^{(k_1)}_{(p',q')}$, and (3) $m^{(k_2)}_{(p,q)} < m^{(k_2)}_{(p',q')}$?
\end{quest}

 Question~\ref{quest:OrderingRelationship} asks about how the orderings in the intersection of the gray zone of two values of $k$ can compare. We weakly expect that no such points can exist, since, if $L(k_2) < \frac{q'-q}{p'-p} < U(k_2)$, then there exists some $k_3 > k_2$ such that $ \frac{q'-q}{p'-p}  < L(k)$ for all $k \geq k_3$. That is, we expect that, once we have $m^{(k)}_{(p,q)} > m^{(k)}_{(p',q')}$, then for all larger $k' > k$, $m^{(k')}_{(p,q)} > m^{(k')}_{(p',q')}$.

\bibliographystyle{abbrv}
\bibliography{bibliography}

\end{document}